\RequirePackage{fix-cm} 

\documentclass[smallcondensed]{svjour3}    

\smartqed              
\renewcommand\qed{\hfill $\blacksquare$}
  
\usepackage{amsfonts}
\usepackage{amsmath} 
\usepackage{subfig}
\usepackage{tikz} 
\usetikzlibrary{shapes.geometric}
\usetikzlibrary{calc}
\usepackage{cite}
\usepackage{xstring} 
\usepackage{enumitem}   
\usepackage{multirow}

\renewcommand{\H}{\mathcal{H}}
\newcommand{\Skx}{S_{i_k}\left(\hat{x}^k\right)}
\newcommand{\lkx}{\lambda_k}
\newcommand{\I}{\mbox{Id}}
\newcommand{\N}{\mathbb{N}} 
\newcommand{\R}{\mathbb{R}} 

\newcommand{\fix}[1]{\mathrm{Fix}(#1)}
\newcommand{\limk}{\lim\limits_{k\rightarrow\infty}}
\newcommand{\tab}{\hspace{20 pt}}
\newcommand\numtrials{30\ }
\newcommand\expepsilon{10^{-2}}

\usepackage{braket}
\usepackage[ruled,vlined]{algorithm2e}
 
\usetikzlibrary{arrows}
\newcommand{\midarrow}{\tikz \draw[-triangle 45] (0,0) -- +(0.01,0);}
 	\newcommand{\rev}[1]{\textcolor{black}{#1}}

\usepackage{amsmath,amssymb,graphicx,stmaryrd}	
\makeatletter
\newcommand{\warrowfull}[2][]{%
	\ext@arrow 0359\warrowfullfill@{#1}{#2}%
}
\newcommand{\warrowfullfill@}{%
	\arrowfill@\relbar\relbar{\mathrel{\smash{\rightharpoonup}\vphantom{\rightarrow}}}%
}

\newcommand{\rarrowfull}[2][]{%
	\ext@arrow 0359\rarrowfullfill@{#1}{#2}%
}

\newcommand{\rarrowfullfill@}{%
	\arrowfill@\relbar\relbar{\mathrel{\smash{\rightarrow}\vphantom{\rightarrow}}}%
}	
\makeatother

\newcommand{\warrow}{\warrowfull[\hspace*{7 pt}]{}}
\newcommand{\rarrow}{\rarrowfull[\hspace*{7 pt}]{}}

\journalname{No Journal.}

\begin{document} 
    
	\title{Asynchronous Sequential Inertial Iterations for Common Fixed Points Problems with an Application to Linear Systems
	}

	\titlerunning{Asynchronous Sequential Inertial Iterations for Common Fixed Points Problems} 
	
	\author{Howard Heaton \and Yair Censor}
	\authorrunning{Howard Heaton \and Yair Censor}
	
	\institute{
		Howard Heaton \at
		Department of Mathematics, 
		University of California Los Angeles, 
		Los Angeles, CA, 90095, USA \\
		\email{heaton@math.ucla.edu}   
		\and		
		Yair Censor \at
		Department of Mathematics, 
		University of Haifa, 
		Mt. Carmel, Haifa, 3498838, Israel \\
		\email{yair@math.haifa.ac.il}      
	}
	
	\date{Received: date / Accepted: date / {\bf Original submission: August 14, 2018. Revised: January 23, 2019.}}
	
	\maketitle
	
	\begin{abstract}
	 The common fixed points problem requires finding a point in the intersection of fixed points sets of  a finite collection of operators.
	 Quickly solving problems of this sort is of great practical importance for engineering and scientific tasks (e.g., for computed tomography).
	 Iterative methods for solving these problems often employ a Krasnosel'ski\u{\i}-Mann type iteration. We present an Asynchronous Sequential Inertial (ASI)  algorithmic framework  in a  Hilbert space  to solve common fixed points problems with a collection of nonexpansive operators. Our scheme allows use of out-of-date iterates when generating updates, thereby enabling processing nodes to work simultaneously and without synchronization.
	 This method also includes inertial type extrapolation terms to increase the speed of convergence.
	 In particular, we extend the application of the recent ``ARock algorithm'' [Peng, Z. et al. SIAM J. on Scientific Computing {\bf 38}, A2851-2879, (2016)] in the context of convex feasibility problems.
	 Convergence of the ASI algorithm is proven with no assumption on the distribution of delays, except that they be uniformly bounded. Discussion is provided along with a computational example showing the performance of the ASI algorithm applied in conjunction with a diagonally relaxed orthogonal projections (DROP) algorithm for estimating solutions to large linear systems. \\

	\end{abstract}

 \section{Introduction}  	
 	In this paper, we investigate common fixed points problems with nonexpansive operators in Hilbert space.
 	In this framework, we focus on the special case of convex feasibility problems.
	Solving convex feasibility problems is of practical interest because it has many applications, including areas in image recovery (e.g., computed tomography (CT)) \cite{1996-combettes-convex}, radiation therapy treatment planning \cite{1988-Censor-Cimmino-IMRT}, electron microscopy, seismology, and others,   see, e.g., \cite{1996-combettes-convex} for a collection of references. Since the speed of individual processing cores  stopped increasing significantly  and multi-core chips are becoming increasingly available \cite{2005-Geer-Chip-to-Multicore}, schemes for solving these problems in parallel are of great use. Notably, these include block-iterative methods  (see, e.g.,  \cite{1989-Aharoni-block-it-parallel-methods,2006-Bausche-Extrapolation,1995-Kiwiel-Block-it-surrogate,2001-Censor-BICAV,1988-Censor-Parallel-Block-It-Methods-Imaging-Therapy} and further references in \cite{2012-Censor-effectiveness})  and string-averaging methods \rev{(see, e.g., \cite{2001-Censor-Averaging,2016-Reich-Modular-String-Averaging})}. \\

	Although parallel methods are being developed to utilize several processing nodes at once, these algorithms are often still expressible, at some level, as sequential algorithms. For example, individual block operators may be computed using several processing nodes in parallel, but these results must be merged together (e.g., via a convex combination) and then passed to the next block operator in the sequence. That is, the collection of block operators are applied successively.
	In this work, we show that such inherently sequential algorithms can be  executed in parallel without the need for synchronization.
	This introduces robustness to dropped network transmissions, introduces a higher level of parallelization, and allows multiple processing nodes to run independently, thereby allowing faster computations of each iterate.
	\\
 
 	{\bf Related Works.}
 	Asynchronous algorithms date back at least to the early work of Chazan and Miranker  \cite{1969-Chazan}, where they used the phrase  ``chaotic relaxations''  (now often termed asynchronous relaxations) for the solution of linear systems. This was followed by numerous works (e.g., Strikwerda \cite{1997-Strikwerda}).
 	For a discussion on asynchronous algorithms, see the summary work of Frommer and Szyld\cite{2000-Frommer}.
 	In Bertsekas and Tsitsiklis \cite{1989-Bertsekas},  the distinction is made that {\it totally asynchronous} algorithms can tolerate arbitrarily large update delays while {\it partially asynchronous} algorithms are not guaranteed to work unless there is an upper bound on these delays. The analysis in \cite{1989-Bertsekas} is both for totally and partially synchronous algorithms. 
 	\\

 	Elsner, Koltracht, and Neumann \cite{1992-Elsner} proved  convergence of a sequence generated by sequential application of partially asynchronous nonlinear paracontractions. 
 	Their work was done in finite-dimensional Euclidean space and used bounded delays of out-of-date information (i.e., partial asynchrony). 
 	This provides direct application to solving linear systems of equations, e.g., through application of Kaczmarz's method \cite{1937-Kaczmarz}, which is also known in the image reconstruction from projections literature as the Algebraic Reconstruction Technique (ART), as it was discovered there by Gordon, Bender, and Herman \cite{1970-ART}.
 	More recently, Peng, Xu, Yan, and Yin \cite{2016-ARock-Original} proposed the ARock algorithm,  which is an asynchronous algorithmic framework. The task at hand there is to find a fixed point of a single separable nonexpansive operator.
 	They accomplished this by generating updates on random blocks of coordinates, using potentially out-of-date information. 
 	Similarly to this original work on ARock \cite{2016-ARock-Original} and the subsequent work by Hannah and Yin \cite{2016-Hannah-Unbounded}, we achieve our results through establishing the monotonicity of a sequence that includes the classical error added to the error introduced from using out-of-date iterates  to perform the updates.
 	\rev {The diversity of the Krasnosel'ski\u{\i}-Mann iteration, to which our work is related as mentioned in the sequel, makes it an interesting and valuable tool, see, e.g., \cite{2018-MiKM} in this journal.
 	Research activity on asynchronous algorithms is receiving ever growing  attention recently as evidenced by some very interesting recent works of Combettes, Eckstein and co-workers \cite{2017-QNE-Iterations-Affine-Hull,2018-Combettes-Async,2018-Johnstone-Projective-Splitting-Async,2018-Johnstone-Projective-Splitting-Rates,2017-Eckstein-Simplified-Block-Op-Splitting-ADMM}}.\\
	\ \\ 	 
 
 	{\bf Advantages of Asynchronous Algorithms.}  		
 	Asynchronous methods have several desirable features. First, each processing node performs its computations independently whereas traditionally when iterative methods are executed using multiple nodes the speed is limited by the slowest node.
 	This means that asynchronous approaches may reduce the average time required to generate successive updates, especially when load-balancing differences do not occur among the nodes.
 	Second, asynchronous methods are robust to failed transmissions through a network. In a synchronous model, if the output from a slave node is sent but does not make it to the master node used to generate a new iterate $x^{k+1}$, then the processing node must perform its computation and send its output again before $x^{k+1}$ can be computed. However, in an asynchronous model, iterates are continually generated and need not be in a specific order.
 	Lastly, this approach may simplify the software code to implement such algorithms since if we are able to estimate the bound on delays, we do not need to keep track of when information is received from each node. Instead, we need only fetch the most recent output, compute the update $x^{k+1}$, and then send this back to the node for the next computation (c.f. Algorithm \ref{alg: ASI-implementation} below).
 	\\ 	
 	
 	{\bf Our Contribution.}
 	This work establishes the convergence of a general partially asynchronous iterative algorithmic framework in Hilbert space for common fixed points problems. 
 	This allows out-of-date iterates to be used when there is an upper bound on the delays. 
 	In practical terms, this enables processing nodes to run in parallel while still executing an inherently sequential algorithm.  
 	And, this allows for processors to be well-utilized even without load-balancing.	 	
 	The present work is distinct from that of \cite{1992-Elsner} since there the operators were paracontractions
 	on finite-dimensional Euclidean space while we use nonexpansive operators in Hilbert space of possibly infinite dimension.
 	Moreover, our algorithmic structure includes an inertial term and our analysis takes a  different approach.   
 	This work is chiefly an extension of the ARock algorithm in the context of convex feasibility problems.
 	The ARock algorithm \cite{2016-ARock-Original,2016-Hannah-Unbounded} and results in \cite{2016-Hannah-Unbounded} generate a sequence stochastically while  the current work need not generate iterations stochastically. 
 	Our work also differs from \cite{1989-Bertsekas} since there fixed points were found for the special case of a nonexpansive operator with respect to the max norm on a Euclidean space. \\ 
    
    {\bf Outline.} In Section \ref{sec: CFP}, we present the framework for the problem at hand and define the notation for asynchrony used in this work.
    Our proposed algorithm is given in Section \ref{sec: ASI-alg}. Section \ref{sec: math-results} presents our mathematical analysis, which leads to the proof of the main result, Theorem \ref{theorem: main-result}.
    We then direct our attention to apply Theorem \ref{theorem: main-result} to solving large systems of linear equations in Section \ref{sec: linear-systems}.
    There we show that the Diagonally Relaxed Orthogonal Projection (DROP) algorithm of Censor, Elfving, Herman, and Nikazad \cite{2008-Censor-DROP}  and Kaczmarz's algorithm \cite{1937-Kaczmarz} for linear systems can be used within the ASI framework and provide a pseudo-code sample implementation of the ASI algorithm with this application in mind.
    A computational example is provided in Section \ref{sec: example} applying the ASI algorithm with DROP to a model CT image reconstruction problem.
    We make some concluding remarks in Section \ref{sec: conclusion}.\\

    \section{Convex Feasibility Problem in a Common Fixed Points Framework}  
    \label{sec: CFP}    
    Let $\{C_i\}_{i=1}^m$ be a finite family of closed convex sets  in a  Hilbert space $\mathcal{H}$ with the inner product $\braket{\cdot,\cdot}$ and norm $\|\cdot\|$, respectively. The associated {\it convex feasibility problem} (CFP) is
    \begin{equation}
    	\mbox{Find\ } x^*  \in C:= \bigcap_{i=1}^m C_i,
    	\label{eq: CFP}
    \end{equation} 
    where $C \neq \emptyset$.
    A common approach to solving (\ref{eq: CFP}) is to generate a sequence of iterates $\{x^k\}_{k\in \N}$ in  $\H$  with $x^1 \in \H$ arbitrary via an iterative process
    \begin{equation}
    	x^{k+1} := Q_k\left(x^k\right),\ \ \mbox{for all $k\in \N$,}
    \end{equation}
    where $\{Q_k \}_{k\in\N}$ is a sequence of operators. 
    In this work, we instead use an iteration of the form
    \begin{equation}
    	x^{k+1} := F_k\left( x^k, \hat{x}^k \right),
    	\ \ \mbox{for all $k\in \N$,}
    \end{equation}
    where $\{ F_k  \}_{k\in\N}$ is a sequence of mappings from $\H\times \H$ to $\H$ and $\hat{x}^k$ either equals $x^k$, or equals a previous iterate $x^{k-j}$ for some $j\in\N$.
    This enables out-of-date information to be used in creating successive iterates. \\ 
     
 	Suppose we have a collection of $w$ processing nodes. For each iteration index $k$, we let $d^k \in \mathbb{Z}_{\geq 0}^w$ give the delay information for each of the nodes. This means that if the last output from the $i$-th node $N_i$ was computed using the iterate $x^{k-j}$, then the $i$-th component of $d^k$ gives this delay amount, i.e., $(d^k)_i = j$. 
 	Let $\{i_k\}_{k\in \N}$ be an index sequence identifying the index of the operator whose output will be used to compute $x^{k+1}$, where $i_k \in \{1,2,\ldots,m\}$ for all $k\in\N$. 
 	Then $\hat{x}^k$ is the last iterate sent to the $i_k$-th node  and we write
 	\begin{equation}
 		\hat{x}^k := x^{k-(d^k)_{i_k}},
 		\ \ \mbox{for all $k\in \N$.}
 	\end{equation}
 	For example, if $k=12$ is the current iterative step, if $i_{12} = 4$, and if the last output from the 4-th node was generated using an iterate that is now 3 iterative steps out-of-date, then
 	\begin{equation}
 		(d^k)_{i_k}
 		= (d_{12})_{4}
 		= 3,
 	\end{equation}
 	and so
 	\begin{equation}
 		\hat{x}^k
 		= x^{k-(d^k)_{i_k}}
 		= x^{12-3}
 		= x^9.
 	\end{equation}
 	This shows that the iterate $x^{13}$ will be computed using $x^{12}$ and $x^9$.  \\
 	
 	The following definitions will be used in the sequel.
 	
 	\begin{definition}
 		Let $\H$ be a Hilbert space and $D\subseteq \H$ be nonempty. Then an operator $T:D\rightarrow\H$ is
 		\begin{enumerate}[label = (\roman*)]
 			\item
 				{\it firmly nonexpansive} if
 				\begin{equation}
 					\forall \ x,y \in D, \ \ \
 					\|T(x)-T(y)\|^2 \leq \braket{x-y, T(x)-T(y)};
 				\end{equation}
 			\item
 				{\it nonexpansive} if $T$ is Lipschitz with constant 1 so that
 				\begin{equation}
 					\forall \ x,y \in D, \ \ \
 					\|T(x)-T(y)\|\leq \|x-y\|;
 					\label{eq: NE-def}
 				\end{equation}
 			\item
 				{\it quasi-nonexpansive} if, denoting its fixed points set $\fix{T} := \{ x\in \H \mid x = T(x)  \}$,
 				\begin{equation}
 					\forall \ x \in D, \ y \in \fix{T}, \ \ \
 					\|T(x)-y\|\leq \|x-y\|.
 					\label{eq: qNE-def}
 				\end{equation}
 		\end{enumerate}
 	\end{definition}
 	
 	We say that $T$ is {\it strictly nonexpansive} 
 	\rev{if the inequality in (\ref{eq: NE-def}) is strict whenever $x-y\neq T(x)-T(y)$, and we say that $T$ is {\it strictly quasi-nonexpansive} if the inequality in (\ref{eq: qNE-def})  is strict whenever $x \notin \fix{T}$.}
 	
 		\rev{More information on firmly nonexpansive operators can be found in the book \cite{1984-Goebel-Book}, see, in particular, page 42 there, and in modern books like \cite{2012-Cegielski-Iterative-methods}.}
 	
 	\begin{definition}
 		\label{def: relaxed-operator}
 		\rev{Let $T:\H\rightarrow\H$ and $\alpha \in [0,2]$. The operator $T_\alpha:\H\rightarrow\H$ defined by
 		$T_\alpha := (1-\alpha) \I + \alpha T$ is called an $\alpha$-relaxation of the operator $T$, where $\I$ is the identity operator.}
 	\end{definition}
 	
 	\begin{definition}
 		\label{def: averaged-operator}
 		\rev{If an operator $Q:\H\rightarrow\H$ is an $\alpha$-relaxation of a nonexpansive operator $T:\H\rightarrow\H$ with $\alpha \in (0,1)$, then $Q$ is said to be $\alpha$-averaged. 
 		And, if the particular value of $\alpha$ is not important, then we simply write that $Q$ is averaged.} 		
 	\end{definition} 
 	
 		\rev{The notion of an averaged operator dates back as early as the work \cite{1978-Bailion-Asymptotic-NE-Mappings} where it was called an averaged mapping, and these terms are both commonly used today.
 		However, for ease of expression in our analysis below, we will often refer to the $\lambda$-relaxation of a nonexpansive operator (which yields an averaged operator for our choices of $\lambda$).
 		}
 	
 	\begin{definition}
 		A {\it  paracontraction} is a continuous strictly quasi-nonexpansive operator.
 	\end{definition}
 	
 		 More information on paracontractions can be found in the paper \cite{1996-Paracontractions-FNE-Operators}. \\
 	
 		The collections of paracontractions and of nonexpansive operators form partially overlapping but distinct subsets of the collection of quasi-nonexpansive operators.
 		Their intersection is nonempty since projections are contained in each of them.
 		However, the identity mapping and reflections are nonexpansive, but are not paracontractions.
 		Additionally, a paracontraction need not be Lipschitz continuous. 
 	
 	The following definition follows \cite[Definition 5.1.1]{1997-Censor-parallel-book}. 
 	
 	\begin{definition}
 		A sequence $\{i_k\}_{k\in \N} $ is called an {\it almost cyclic control} on 
 		$[m] := \{1,2,\ldots,m\}$ if $i_k\in [m]$ for all $k\in\N$ and there exists an integer $M\geq m$ (called the {\it almost cyclicality constant}) such that, for each $k\in \N$,  $[m] \subseteq \{i_{k+1},i_{k+2},\ldots,i_{k+M}\}$.
 	\end{definition}
 	 
 	\section{The Asynchronous Sequential Inertial (ASI) Algorithm}  
 	\label{sec: ASI-alg}
 	We now describe  our proposed algorithm.
 	Consider a collection of $m$ nonexpansive operators $\{T_i\}_{i=1}^m$ on $\H$ with a common fixed point. We seek to solve (\ref{eq: CFP}) where $C_i = \fix{T_i}$ for all $i \in [m].$ 
 	For notational convenience, we define
 	\begin{equation}
 		S_i := \I - T_{i},
 		\ \ \mbox{for all $i\in[m]$.}
 		\label{eq: Original-Si-def}
 	\end{equation}
 	Our Asynchronous Sequential Inertial (ASI) algorithm is as follows.\\
 	 
	\begin{algorithm}[H]
		\caption{Asynchronous Sequential Inertial (ASI) Algorithm \label{alg: ASI}}
		\normalsize
		Let $x^1 \in \H$ be arbitrary,  $\{\lambda_k\}_{k\in\N}$ be such that $\lambda_k \in (0,1)$ for all $k\in\N$, and $\{i_k\}_{k\in\N}$ be an almost cyclic control on $[m]$. 
		For each $k\in \N$ set
		\begin{equation}
			x^{k+1}
			:= \begin{cases}
			\begin{array}{ll}
			x^k, & \mbox{if $k \leq \sup_{k\in\N}\|d^k\|_\infty$,} \\
			x^k - \lambda_k S_{i_k}\left(\hat{x}^k\right),\ \ & \mbox{otherwise.}
			\end{array}
			\end{cases}
			\label{eq: ASI-algorithm-iteration}
		\end{equation}
	\end{algorithm} 	
 	\ \\
 	
 	In the special case where $\hat{x}^k = x^k$ \rev{and $\lambda_k = \lambda$} for each $k\in \N$ and we have a single operator $T$ (i.e., $m=1$), we obtain
    \begin{equation}
    	x^{k+1} 
    	:= T_\lambda\left(x^k\right)
    	=(1-\lambda) x^k + \lambda T \left(x^k\right),
    	\label{eq: sync-iteration}
    \end{equation}
    which is  precisely the Krasnosel'ski\u{\i}-Mann (KM) iteration, see, e.g., \cite{1953-Mann,1955-Krasnoselskii} for the original works, \rev{Section 5.2 of the book \cite{2017-Bausche-Combettes} for a summary, and \cite{1979-Reich-Weak-Convergence-Theorems} for more information on the KM iteration.}    
    This iteration does not allow any delays and  generates a sequence that weakly converges to a fixed point of $T$.
    The primary result of our current work is Theorem \ref{theorem: main-result}, which states that a sequence $\{x^k\}_{k\in\N}$ generated by the ASI algorithm converges weakly to a solution of (\ref{eq: CFP}) when the sequence of delays $\{d^k\}_{k\in \N}$ is uniformly bounded in the sup norm by some $\tau\geq 0$ and the step sizes $\{\lambda_k\}_{k\in\N}$ are bounded above by $1/(2\tau+1)$. \\

    Before presenting our analysis, we provide a remark and illustration of the ASI algorithm to give the reader some intuition.  The computation of $x^{k+1}$ can be expressed in two parts. The first part is a convex  combination of $x^k$ and $T_{i_k}(\hat{x}^k)$ to form the point
    \begin{equation}
    	y^k := (1-\lambda_k) x^k + \lambda_k T_{i_k}\left(\hat{x}^k\right). 
    \end{equation}
    The second part is an inertial term that estimates the direction of the solution, given $x^k$ and $\hat{x}^k$. 
    The term $y^k$ is added to the inertial term to yield $x^{k+1}$, i.e.,
    \begin{equation}
    	\begin{aligned}
	    	x^{k+1} 
	    	&= x^k - \lambda_k S_{i_k} \left(\hat{x}^k\right) \\
	    	& = x^k - \lambda_k \left( \hat{x}^k - T_{i_k}\left(\hat{x}^k\right)\right)\\
	    	&= \underbrace{(1-\lambda_k) x^k + \lambda_k T_{i_k}\left(\hat{x}^k\right)}_{\mbox{convex combination}} + \underbrace{\lambda_k \left(x^k - \hat{x}^k  \right).}_{\mbox{inertial term}}
	    	\label{eq: update-lin-comb-plus-pert}
    	\end{aligned} 
    \end{equation} 
	Since the iterates are, on average, moving closer to a solution, this inertial term may accelerate the convergence by using previous information along with the current iterate to estimate the direction toward a solution. 
	This effect of inertial terms is known in the literature, see, e.g., \cite{2001-Alvarez,2003-Moudafi-Inertial-Terms,2008-Mainge-KM-Inertial,1964-Polyak,2015-Lorenz-Inertial-Terms}. 		
	The right-hand side of (\ref{eq: update-lin-comb-plus-pert}) also illustrates the distinction of ASI from the algorithm of \cite{1992-Elsner}, which uses only the convex combination part of (\ref{eq: update-lin-comb-plus-pert}) without the inertial term. \\
    
    To illustrate  the ASI algorithm graphically, let $C_1$ and $C_2$ be two closed convex sets with nonempty intersection and let $T_1$ and $T_2$ be relaxations of the projections $P_{C_1}$ and $P_{C_2}$ onto the sets $C_1$ and $C_2$, respectively. 
    In this case, Figure \ref{fig: ASI-algorithm} shows how $x^{k+1}$ is generated from $x^k$ and $\hat{x}^k$. 

   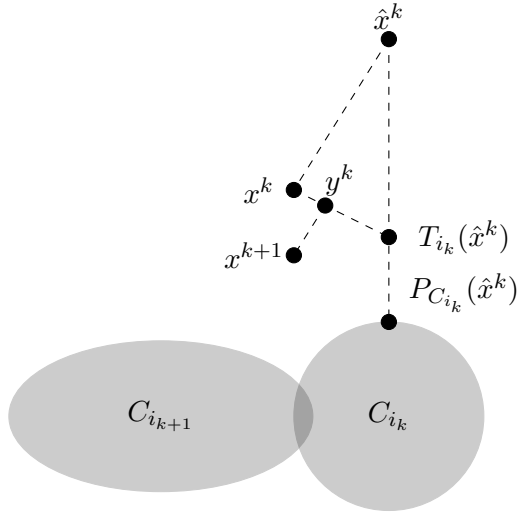
\begin{figure} 
   	\centering 
   	{\large
   		\begin{tikzpicture} 
   		\def\P{1.0}
   		\draw[fill = black, opacity = 0.20] (2,0) circle (1.25*\P);	   		
   		\draw[fill = black, opacity = 0.20] (-1,0) ellipse (2 and 1);
   		 
   		\coordinate (x1) at (2,5);
   		\coordinate (x2) at (0.75,3.0);
   		\coordinate (x3) at (2,1.25);
   		\coordinate (x4) at ($(x1)!0.70!(x3)$);	   		
   		\coordinate (x5) at ($(x2)!0.33!(x4)$);	  
   		\coordinate (pert) at ($(x2)-(x1)$);
   		\coordinate (x6) at ($(x5)+0.33*(pert)$);
   		
   		\draw[dashed] (x1) -- (x3);
   		\draw[dashed] (x1) -- (x2);
   		\draw[dashed] (x2) -- (x4); 
   		\draw[dashed] (x5) -- (x6); 
   		
   		\draw[fill = black] (x2) circle (0.1) node[left] {${x}^k\   $};
   		\draw[fill = black] (x4) circle (0.1) node[right] {$\ \ T_{i_k}(\hat{x}^k)$};	  
   		\draw[fill = black] (x1) circle (0.1) node[above] {$\hat{x}^k$};		
   		\draw[fill = black] (x3) circle (0.1) node[above right] {$\   P_{C_{i_k}}(\hat{x}^k)$};
   		
   		\draw[fill = black] (x5) circle (0.1) node[above] {\ \ \ $y^k$};
   		\draw[fill = black] (x6) circle (0.1) node[left] {$x^{k+1}$\   };  	
   		
   		\draw[] ( 2,0) node[] {$C_{i_k}$};	   			
   		\draw[] (-1,0) node[] {$C_{i_{k+1}}$};	
   		\end{tikzpicture}
   	}
   	\caption{Illustration of a full iteration of the ASI algorithm with two convex sets and the $T_i$'s as relaxed projections onto the sets.}
   	\label{fig: ASI-algorithm}
   \end{figure}
 
    \section{Mathematical Analysis of the ASI Algorithm}
    \label{sec: math-results}
    In this section, we provide several lemmas  culminating in our primary convergence result in Theorem \ref{theorem: main-result}. 
    We begin with the following lemma on  the {\it demi-closedness principle}, which is a slight generalization of Cegielski \cite[Lemma 3.2.5]{2012-Cegielski-Iterative-methods}.

    \begin{lemma}
    	\label{lemma: weak-cluster-fixed-point}
    	Let $\{T_i \}_{i=1}^m$ be a finite family of nonexpansive operators with a common fixed point and let $y\in \H$ be a weak cluster point of a sequence $\{x^k\}_{k\in \N}$. If $\|T_i x^k-x^k\|\rarrow 0$ for all $i \in [m]$, then
    	\begin{equation}
    		y \in \bigcap_{i=1}^m \fix{T_i}.
    	\end{equation}
    \end{lemma}
    \begin{proof}
    	By hypothesis, there is a subsequence $\{x^{n_k}\}_{k\in\N} \subseteq \{x^k\}_{k\in\N}$ such that $x^{n_k}\warrow y$.
    	Pick any $j \in [m]$.    	
    	Then, using the triangle inequality, we deduce
    	\begin{equation}
    		\begin{aligned}
	    		\liminf_{k\rightarrow\infty}
	    		\| x^{n_k} - y\|
	    		&\geq 
	    		\liminf_{k\rightarrow\infty}
	    		\|T_j x^{n_k} -T_j y\| \\
	    		& = \liminf_{k\rightarrow\infty}
	    		\|T_j x^{n_k} - x^{n_k} + x^{n_k}- T_j y\|  \\
	    		& \geq \liminf_{k\rightarrow\infty}
	    		\left(\| x^{n_k}- T_j y\| - \|T_j x^{n_k} - x^{n_k}\|    \right) \\
	    		& = \liminf_{k\rightarrow\infty}
	    		 \| x^{n_k}- T_j y\| . \\	    		
    		\end{aligned} 		
    		\label{eq: lemma-1-1}
    	\end{equation}
    	A 1967 lemma of Opial, see, e.g.,	\cite[Lemma 3.2.4]{2012-Cegielski-Iterative-methods}, states that if $z^k \warrow z \in \H$ and $z' \in \H$ with $z\neq z'$, then
    	\begin{equation}
    		\liminf_{k\rightarrow\infty} \|z^k - z'\|
    		>\liminf_{k\rightarrow\infty} \|z^k - z\|.
    		\label{eq: lemma-1-2}
    	\end{equation}
    	Consequently, if $T_jy \neq y$, then (\ref{eq: lemma-1-1}) and (\ref{eq: lemma-1-2}) together imply
    	\begin{equation}
    		\liminf_{k\rightarrow\infty}
    		\| x^{n_k} - y\|
    		\geq  \liminf_{k\rightarrow\infty}
    		\| x^{n_k}- T_j y\|
    		> \liminf_{k\rightarrow\infty}
    		\| x^{n_k} - y\|,
    	\end{equation}
    	a contradiction. Whence, $T_j y = y$ for each $j$ and the result follows.
    	\qed
    \end{proof}
    
    This lemma helps obtain the following proposition, which is key in obtaining our main result. 
    The second half of the proof of Proposition \ref{proposition: weak-convergence} is based on the result found in \cite[Corollary 3.3.3]{2012-Cegielski-Iterative-methods}, which is credited there  to Bauschke and Borwein \cite[Theorem 2.16(ii)]{1996-bauschke-projection}.
    This corollary in \cite{2012-Cegielski-Iterative-methods} established that Fej\'er monotone sequences with respect to a closed set have at most one weak cluster point. Their proof idea  is modified below to apply to the current setting. \rev{In connection to this proof, we refer the reader also to the papers \cite{1979-Reich-Weak-Convergence-Theorems} and \cite{1983-Reich-Note-Mean-Ergodic-Theorem}.}
    
	\begin{proposition}
		\label{proposition: weak-convergence}
		Let $\{T_i\}_{i=1}^m$ be a family of nonexpansive operators on $\H$ with a common fixed point and let $\{x^k\}_{k\in\N}$ be a sequence in $\H$.
		If  for every
		\begin{equation}
			z \in C := \bigcap_{i=1}^m \fix{T_i} 
		\end{equation}
		the sequence $\{\|x^k-z\|\}_{k\in\N}$ converges and
		$\|T_i x^k-x^k\|\rarrow 0$ for all $i\in [m]$, then the sequence
		$\{x^k\}_{k\in\N}$ converges weakly to some $x^* \in C$.
	\end{proposition}
	
	\begin{proof}
		Let $z\in C$.
		The triangle inequality and the convergence of $\{\|x^k-z\|\}_{k\in\N}$   imply that the sequence $\{x^k\}_{k\in\N}$ is bounded. Consequently, it has a weakly convergent subsequence $\{x^{n_k}\}_{k\in\N}\subseteq \{x^k\}_{k\in\N}$.
		For each such subsequence, we may apply Lemma \ref{lemma: weak-cluster-fixed-point} to assert the weak cluster point is contained in  $C$.
		We show below that the sequence $\{x^k\}_{k\in\N}$ has at most one weak cluster point in $C$.
		Thus, there is precisely one weak cluster point of $\{x^k\}_{k\in\N}$ and it is contained in $C$, from which the result will follow. \\
		
		It remains to verify that there is at most one cluster point of $\{x^k\}_{k\in\N}$ in $C$. 
		Define the function $f:C\rightarrow\R$ by
		\begin{equation}
			f(z) := \limk \left( \|x^k-z\|^2-\|z\|^2 \right),
		\end{equation}
		which, by hypothesis, converges for each $z \in C$. 
		Note that
		\begin{equation}
			 \|x^k-z\|^2-\|z\|^2
			= \left(\|x^k\|^2 - 2\braket{x^k,z} + \|z\|^2\right) -\|z\|^2
			= \|x^k\|^2 - 2\braket{x^k,z}.
		\end{equation}
		Now let $\{x^{m_k}\}_{k\in\N}$ and $\{x^{n_k}\}_{k\in\N}$ be subsequences of $\{x^k\}_{k\in\N}$ weakly converging to two distinct limit points $p$ and $q$ in $C$, respectively. 
		Then 
		\begin{equation}
			\begin{aligned}
				2\limk \braket{x^k, p-q}
				&= \limk \Big(\left( \|x^k\|^2 - 2\braket{x^k,q} \right)-\left( \|x^k\|^2 - 2\braket{x^k,p} \right) \Big)\\
				&= f(q) - f(p),
			\end{aligned}
			\label{eq: prop-1}
		\end{equation}
		i.e., the limit exists. However,
		\begin{equation}
			\limk \braket{x^{m_k}, p-q}
			= \braket{p,p-q},
			\label{eq: prop-2}
		\end{equation}
		and
		\begin{equation}
			\limk \braket{x^{n_k}, p-q}
			= \braket{q,p-q}.
			\label{eq: prop-3}
		\end{equation}
		Since the limit in (\ref{eq: prop-1}) exists, the limits in (\ref{eq: prop-2}) and (\ref{eq: prop-3}) must be equal, which implies
		\begin{equation}
			\braket{p,p-q} = \braket{q,p-q}
			\ \ \ \Longrightarrow \ \ \ 
			\|p-q\|^2 = 0
			\ \ \ \Longrightarrow \ \ \ 
			p = q.
		\end{equation}
		This shows the weak cluster point is unique and completes the proof. 
		\qed
	\end{proof} 
 	The above proposition is essential for our convergence result.
 	All our subsequent lemmas compiled together show that the assumptions of Proposition \ref{proposition: weak-convergence} hold for sequences generated by the ASI algorithm. \\
 	
    We outline the remainder of this section as follows.    
    We first give in Lemma \ref{lemma: fundamental-inequality}  a fundamental inequality about sequences $\{x^k\} _{k\in\N}$ generated by the ASI algorithm.
    This inequality is used in Lemma \ref{lemma: xi-convergence} to show that the sequence $\{\xi_k\}_{k\in\N}$ converges, where $\xi_k$ is a sum of the classical distance $\|x^k-z\|$ for some $z\in C$ and finitely many terms of the form $c_i \|x^{k+1-i}- x^{k-i}\|\rev{^2}$. 
    We then use the inequality  (\ref{eq: xik-inequality}) in the proof of this lemma to verify that $\|x^{k+1}-x^k\|\rarrow 0$ and that the sequence $\{\|x^k-z\|\}_{k\in\N}$ converges for each $z\in C$. 
    Following this,  in Lemma \ref{lemma: T-operator-residual-convergence} we prove that $\|T_ix^k-x^k\|\rarrow 0$ for each $i\in[m]$.
    With these results, we  show that the hypotheses of Proposition \ref{proposition: weak-convergence} hold for  any sequence $\{x^k\} _{k\in\N}$ generated by the ASI algorithm. 
 	\rev{In what follows, we always make the following assumption.} \\
 	\rev{\newline{\bf Assumption 1}
 			The delay vectors $\{d^k\}_{k\in\N}$ are uniformly bounded in sup norm by some $\tau \geq 0$.}

    \begin{lemma}
    	\label{lemma: fundamental-inequality}
    	Let  $z \in C$ and let $\mu > 0$ \rev{and suppose Assumption 1 holds. 
    	If $\{x^k\}_{k\in\N}$ is any sequence generated by the ASI algorithm, then}  
    	\begin{equation}
    		\begin{aligned}
	    		\left\|x^{k+1}-z\right\|^2
	    		& \leq \left\|x^k-z\right\|^2 + \mu \sum_{\ell=1}^\tau   \left\|x^{k+1-\ell}-x^{k-\ell}\right\|^2 \\
	    		&  - \lambda_k \left\|S_{i_k}\left(\hat{x}^k\right)\right\|^2\Big( 1 - \lambda_k \left( 1+ \tau/\mu \right)  \Big).
    		\end{aligned}
    	\end{equation}
    \end{lemma} 
    \begin{proof}
    	First observe that
    	\begin{equation}
    		\begin{aligned}
	    		\|x^{k+1}-z\|^2
	    		&= \|x^{k} - \lambda_k \Skx - z\|^2\\
	    		&= \|x^k - z\|^2 - 2\lkx\braket{x^k-z,\Skx} + \lkx^2 \|\Skx\|^2.
    		\end{aligned}
    		\label{eq: xk-expanded}
    	\end{equation}
    	We may split the cross-term into the two expressions $\alpha_k$ and $\beta_k$ via
    	\begin{equation}
    		-2\lambda_k \braket{\Skx , x^k-z}
    		= \underbrace{-2\lambda_k \braket{\Skx, \hat{x}^k-z}}_{\alpha_k} + \underbrace{  2\lambda_k \braket{\Skx, x^k - \hat{x}^k} }_{\beta_k}.
    	\end{equation}
    	Note that $\frac{1}{2}S_{i_k} = \frac{1}{2}\left(\I-T_{i_k}\right)$ is firmly nonexpansive and that $\frac{1}{2}S_{i_k}(z)  = 0$.
    	Thus,
    	\begin{equation} 
    		\begin{aligned}
    			\|\Skx\|^2
    			& = 4 \|\frac{1}{2}\Skx\|^2 \\
    			& = 4 \|\frac{1}{2}\Skx - \frac{1}{2}S_{i_k}(z)\|^2 \\
    			& \leq 4 \braket{\frac{1}{2}\Skx - \frac{1}{2}S_{i_k}(z), \hat{x}^k-z} \\
    			& = 2\braket{\Skx, \hat{x}^k-z}.
    		\end{aligned}
    	\end{equation}
    	This implies
    	\begin{equation}
    		\alpha_k = -2\lambda_k \braket{\Skx, \hat{x}^k-z}
    		\leq -\lambda_k \|\Skx \|^2.
    		\label{eq: 1-alpha-inequality}
    	\end{equation}
    	Application of the triangle inequality and the fact that $\|d^k\|_\infty\leq \tau$ for all $k\in\N$ yield
    	\begin{equation}
    		\begin{aligned}
    			\beta_k
    			& =- 2\lkx \braket{\Skx,x^k-\hat{x}^k} \\
    			& = - 2\lkx \sum_{\ell=1}^{(d^k)_{i_k}} \braket{\Skx, x^{k+1-\ell} - x^{k-\ell}} \\
    			& \leq   2\lkx \sum_{\ell=1}^{(d^k)_{i_k}} \|\Skx\| \|x^{k+1-\ell} - x^{k-\ell}\| \\
    			& \leq  2\lkx \sum_{\ell=1}^{\tau} \|\Skx\| \|x^{k+1-\ell} - x^{k-\ell}\|.    			
    		\end{aligned}
    		\label{eq: fund-inequality-proof-1}
    	\end{equation}
    	Note that the second line in (\ref{eq: fund-inequality-proof-1}) holds since the sum is telescoping and that the final line holds since $\|d^k\|_\infty \leq \tau$.
    	Using the fact that $0 \leq (a-b)^2 = a^2+b^2-2ab$ for $a,b\in\R$ implies $ab \leq \frac{1}{2}(a^2+b^2)$, we deduce
    	\begin{equation}
    		\begin{aligned}
    			\beta_k
    			& \leq   \lkx \sum_{\ell=1}^{\tau} \dfrac{\lkx}{\mu}\|\Skx\|^2 + \dfrac{\mu}{\lkx} \|x^{k+1-\ell} - x^{k-\ell}\|^2 \\ 
    			& = \dfrac{\tau \lkx^2}{\mu } \|\Skx\|^2   + \mu \sum_{\ell=1}^\tau  \|x^{k+1-\ell} - x^{k-\ell}\|^2.
    		\end{aligned}
    		\label{eq: 2-beta-inequality}
    	\end{equation}
    	Combining (\ref{eq: xk-expanded}), (\ref{eq: 1-alpha-inequality}) and (\ref{eq: 2-beta-inequality}), we obtain the desired result.    	
    	\qed
    \end{proof}
    
    To prove the convergence of iterates to solutions of fixed point problems, we typically need some sort of monotonicity. However, when using out-of-date iterates to generate the sequence, the inequality $\|x^{k+1}-z\|\leq \|x^k-z\|$ does {\it not} necessarily hold. 
    In \cite{2016-Hannah-Unbounded}, the authors were able to construct a sequence that is monotonic in expectation by   including both the classical error $\|x^k-z\|$ and terms of the form $c_\ell \|x^{k+1-\ell}-x^{k-\ell}\|$. We also use the idea of adding terms of this form with the inequality in Lemma \ref{lemma: fundamental-inequality} to obtain convergence  of a sequence that sums the classical error and a finite number of these discrepancy terms. This is stated formally in the following lemma.
  
 	\begin{lemma}
 		\label{lemma: xi-convergence}
 		Let $z\in C$ and let $\mu > 0$ \rev{and suppose Assumption 1 holds.
 		Let $\{x^k\}_{k\in\N}$ be any sequence generated by the ASI algorithm.
 		Assume there is $\varepsilon > 0$ such that}
 		\begin{equation}
 		0 <  \varepsilon \leq  \lambda_k \leq \dfrac{1}{1 + \tau(1/\mu + \mu) + \varepsilon},
 		\ \ \mbox{for all $k\in\N$.} 
 		\label{eq: lambda-step-size-bd-mu}
 		\end{equation} 		 		
 		The sequence $\{\xi_k\}_{k\in\N}$ defined by
 		\begin{equation}
 			\xi_k := \|x^k-z\|^2 + \sum_{\ell=1}^\tau c_\ell \|x^{k+1-\ell}-x^{k-\ell}\|^2,
 			\label{eq: xi-k-def}
 		\end{equation} 
 		where
 		\begin{equation}
 			c_j : = (\tau + 1 - j)\mu + \varepsilon,
 			\ \ \ \mbox{for all $j\in [\tau+1]$,}			
 		\end{equation}
 		is a convergent sequence.			
 	\end{lemma}   
 	\begin{proof}
 		The sequence $\{\xi_k\}_{k\in\N}$ is nonnegative; thus, it suffices to verify that it is monotonically decreasing.
 		By Lemma \ref{lemma: fundamental-inequality}, we have  
 		\begin{equation}
 			\begin{aligned}
 				\xi^{k+1}
 				& = \|x^{k+1}-z\|^2 + \sum_{\ell=1}^\tau c_\ell \|x^{k+2-\ell}-x^{k+1-\ell}\|^2 \\
 				& \leq \|x^k-z\|^2 + \sum_{\ell=1}^\tau \mu \|x^{k+1-\ell}-x^{k-\ell}\|^2 \\
 				& - \lambda_k \|\Skx\|^2\Big(1-\lambda_k(1+\tau/\mu)\Big) + \sum_{\ell=1}^\tau c_\ell \|x^{k+2-\ell}-x^{k+1-\ell}\|^2.
 			\end{aligned}
 		\end{equation}
 		Reindexing the sums and noting $\mu + c_{j+1} = c_j$, this simplifies to
 		\begin{equation}
 			\begin{aligned}
	 			\xi^{k+1}
	 			& \leq \|x^k-z\|^2 + \sum_{j=1}^\tau \left(\mu + c_{j+1}\right)\|x^{k+1-j}-x^{k-j}\|^2 + c_1\|x^{k+1}-x^k\|^2  \\
	 			& - \lambda_k \|\Skx\|^2\Big(1-\lambda_k(1+\tau/\mu)\Big) - c_{\tau+1} \|x^{k+1-\tau} - x^{k-\tau}\|^2 \\[10 pt]
	 			& = \xi_k + c_1\|x^{k+1}-x^k\|^2 - c_{\tau+1} \|x^{k+1-\tau} - x^{k-\tau}\|^2 \\
	 			& - \lambda_k \|\Skx \|^2\Big( 1-\lambda_k(1+\tau/\mu) \Big).
 			\end{aligned}
 		\end{equation}
 		The ASI algorithm generates successive iterates such that
 		\begin{equation}
 			\|x^{k+1}-x^k\|^2
 			\leq \|x^k-\lkx\Skx - x^k\|^2
 			= \lkx^2 \|\Skx\|^2.
 		\end{equation}
 		Consequently,
 		\begin{equation}
 			\xi_{k+1}
 			\leq \xi_k - \lkx \|\Skx\|^2 \Big( 1-\lkx(1+\tau/\mu + c_1) \Big) - c_{\tau+1} \|x^{k+1-\tau} - x^{k-\tau}\|^2.
 			\label{eq: xik-inequality}
 		\end{equation}
 		With our assumption in (\ref{eq: lambda-step-size-bd-mu}), the fact $c_1 = \tau \mu + \varepsilon$, and the fact $c_{\tau+1} =\varepsilon > 0$,  (\ref{eq: xik-inequality}) shows that $\xi_{k+1}\leq \xi_k$, for all $k\in \N$, from which the result follows.
 		\qed
 	\end{proof}

	\begin{remark}
	 	Naturally, we may seek to choose $\mu$ that maximizes the allowable step sizes.
	 	Such a choice of $\mu$ minimizes the function $f:(0,\infty)\rightarrow\R$ defined by $f(x):= x + 1/x$.  Since the single critical point and minimizer of $f$ occurs at $x=1$, choosing $\mu = 1$ in Lemma \ref{lemma: xi-convergence} yields the optimal step size bound. 
	\end{remark}

 	Although the above lemma  establishes the convergence of $\{\xi_k\}_{k\in\N}$, it does not guarantee that a sequence generated by terms of the form $\|x^k-z\|$ will converge. However, we are able to verify this in the following lemma by noting the inequality in (\ref{eq: xik-inequality}).    
  
 	\begin{lemma}
 		\label{lemma: xk-z-convergence}
 		If the assumptions of Lemma \ref{lemma: xi-convergence} hold,
 		then $\|x^{k+1}-x^k\|\rarrow 0$ and the sequence $\{\|x^k-z\|\}_{k\in \N}$ converges.
 	\end{lemma}
 	\begin{proof}
 		The inequality (\ref{eq: xik-inequality}) implies
 		\begin{equation}
 		0 \leq  c_{\tau+1} \|x^{k+1-\tau} - x^{k-\tau}\|^2
 		\leq \xi_k - \xi_{k+1}.
 		\end{equation} 
 		Letting $k \rarrow \infty$, the convergence of $\{\xi_k\}_{k\in\N}$ together with the squeeze (sandwich) theorem implies
 		\begin{equation}
 		\limk c_{\tau+1} \|x^{k+1-\tau} - x^{k-\tau}\|^2 = 0.
 		\end{equation}
 		Since $c_{\tau+1} = \varepsilon > 0$, we obtain
 		\begin{equation}
 		\limk \|x^{k+1} - x^k \| = 0.
 		\label{eq: xk-difference-converge-to-zero}
 		\end{equation}
 		Let $\xi_*$ be the limit of $\{\xi_k\}_{k\in\N}$.
 		Then combining (\ref{eq: xi-k-def}) and (\ref{eq: xk-difference-converge-to-zero}) and letting $k\rarrow\infty$ yields
 		\begin{equation}
 			\xi_* = \limk \xi_k = \limk \|x^k - z\|^2,
 		\end{equation}
 		from which the result follows since the square root function is continuous. 	
 		\qed
 	\end{proof}

   \begin{lemma}
   	If a sequence $\{x^k\}_{k\in\N}$ satisfies $\|x^{k+1}-x^k\|\rarrow 0$ and \rev{Assumption 1 holds}, then $\|x^k - \hat{x}^k\|\rarrow 0$.
   \end{lemma}
   \begin{proof} 
	The fact that $\|d^k\|_\infty \leq \tau$, for all $k\in\N$, yields
	\begin{equation}
		0 \leq \|x^k - \hat{x}^k\|
		= \|x^k - x^{k-(d^k)_{i_k}}\|
		\leq \sum_{j=1}^\tau \|x^{k+1-j} - x^{k-j}\|.
		\label{eq: xk-hatxk-inequality} 
	\end{equation}
	Letting $k\rarrow\infty$ in (\ref{eq: xk-hatxk-inequality}), the right-hand side goes to zero since the sum contains finitely many terms. The result is then obtained through the squeeze theorem.
	\qed
   \end{proof}
   
   \begin{lemma}
   	\label{lemma: T-operator-residual-convergence}
   	Let $\{x^k\}_{k\in\N}$ be a sequence generated by the ASI algorithm. 
   	\rev{Suppose} there is $\varepsilon > 0$ such that $\lambda_k \geq \varepsilon$, for all  $k\in\N$\rev{, and Assumption 1 holds.
   	If $\|x^{k+1}-x^k\|\rarrow 0$, then $\|T_ix^k - x^k\|\rarrow 0$, for each $i\in[m]$.}
   \end{lemma}
   \begin{proof} 
   	Let $i \in [m]$ and let $T_{i,\lambda}$ be the $\lambda$-relaxation of $T_i$. 	If $\{t_k\}$ is an increasing sequence such that $t_k \in \N$, for all $k\in\N$, then
   	\begin{equation}
   		\|T_{i,\lambda_{t_k}}(x^k)-x^k\|
   		= \|(1-\lambda_{t_k})x^k + \lambda_{t_k} T_i (x^k) - x^k\|
   		= \lambda_{t_k} \|T_i(x^k)-x^k\|
   		\geq  \varepsilon \|T_i(x^k)-x^k\|.
   	\end{equation}
   This inequality reveals that if $\|T_{i,\lambda_{t_k}}(x^k)-x^k\|\rarrow 0$, then 
   	$\|T_i(x^k)-x^k\|\rarrow 0$.
   	Next, we verify that the first limit holds, by setting, for each $k\in\N$,  $t_k$ to be the smallest index greater than or equal to  $k$ such that $i_{t_k} = i$. 
   	This implies
   	\begin{equation}
   		\begin{aligned}
	   		x^{t_k+1}
	   		&= x^{t_k} - \lambda_{t_k} S_{i} \left(\hat{x}^{t_k}\right)\\
	   		&= x^{t_k} - \lambda_{t_k} \hat{x}^{t_k} + \lambda_{i} T_{i} \left(\hat{x}^{t_k}\right) \\
	   		&= T_{i,\lambda_{t_k}} \left(\hat{x}^{t_k}\right)
	   		+  \left( \hat{x}^{t_k} - \hat{x}^{t_k} \right).
   		\end{aligned}
		\label{eq: 1}
   	\end{equation}
	Observe also that
   	\begin{equation}
   		\begin{aligned}
	   		\|T_{i,\lambda_{t_k} }\left(x^k\right) - x^{t_k+1} \|
	   		& = \| T_{i,\lambda_{t_k} }\left(x^k\right) -  T_{i,\lambda_k} \left(\hat{x}^{t_k}\right)
	   		- \left( x^{t_k} - \hat{x}^{t_k} \right) \| \\
	   		& \leq \|T_{i,\lambda_{t_k}} \left(x^k\right) - T_{i,\lambda_{t_k}}  \left(\hat{x}^{t_k}\right) \| +   \| x^{t_k} - \hat{x}^{t_k}  \| \\
	   		& \leq \| x^k - \hat{x}^{t_k} \| +  \|x^{t_k} - \hat{x}^{t_k} \|,
   		\end{aligned}
   		\label{eq: 2}
   	\end{equation} 
   	where the first equality holds by (\ref{eq: 1}) and the final inequality follows from the fact the $\lambda_{t_k}$-relaxation of a nonexpansive operator $T_i$ is nonexpansive.
   	Due to the bound on delays and the almost cyclicity of $\{i_k\}_{k\in\N}$, we know that $\hat{x}^{t_k} = x^j$ for some $ k-\tau \leq j \leq t_k \leq k +M$, where $M$ is the almost cyclicality constant of $\{i_k\}_{k\in\N}$. Thus, repeated application of the triangle inequality with (\ref{eq: 2}) yields 
   	\begin{equation} 
   		\begin{aligned}
   			\|T_{i,\lambda_{t_k} }\left(x^k\right) - x^{t_k+1} \|
   			&  \leq \| x^k - \hat{x}^{t_k} \| + \|x^{t_k} - \hat{x}^{t_k} \|\\
   			& \leq \sum_{\ell=-\tau}^{M-1}  \|x^{k+\ell+1} - x^{k+\ell} \|
   			+  \sum_{\ell=-\tau}^{M-1} \|x^{k+\ell+1} - x^{k+\ell} \| \\
   			& = 2 \sum_{\ell=-\tau}^{M-1}  \|x^{k+\ell+1} - x^{k+\ell} \|.
   		\end{aligned}
   	\end{equation} 
   	In a similar fashion, we deduce 
   	\begin{equation} 
   		\begin{aligned}
   			\| T_{i,\lambda_{t_k}} (x^k) - x^k\|
   			& \leq \| T_{i,\lambda_{t_k}}(x^k) - x^{t_k+1} \| + \|x^{t_k+1} - x^k\| \\
   			& \leq 2 \sum_{\ell=-\tau}^{M-1}  \|x^{k+\ell+1} - x^{k+\ell} \| + \|x^{t_k+1} - x^k\| \\
   			& \leq 3 \sum_{\ell=-\tau}^{M-1}  \|x^{k+\ell+1} - x^{k+\ell} \|.
   		\end{aligned}
   	\end{equation}  
   	Letting $k\rarrow\infty$, the right-hand side goes to zero since it is the sum of a finite number of terms that, by hypothesis, converge to zero.
   	This verifies $\|T_{i,\lambda_k}(x^k) - x^{k}\|\rarrow 0$, from which the result follows. 
   	\qed
   \end{proof}

   Now we can state and prove the main result of this paper.
   As noted previously, this result is about a generalization of the procedure that successively applies the $\lambda_k$-relaxation of the operator $T_{i_k}$, for all $k\in\N$, which occurs when $\tau = 0$.

    \begin{theorem} 
    	\label{theorem: main-result}
    	Let  $\{x^k\}_{k\in\N}$ be a sequence generated by the ASI algorithm \rev{and suppose Assumption 1 holds}.
    	If there is $\varepsilon > 0$ such that
    	\begin{equation}
    		0 < \varepsilon \leq \lambda_k \leq \dfrac{1}{2\tau + 1 + \varepsilon},
    		\ \ \mbox{for all $k\in\N$,}
    	\end{equation}
    	then the sequence $\{x^k\}_{k\in\N}$ converges weakly to a common fixed point $x^*$ of the family $\{T_i\}_{i=1}^m$, i.e.,
    	\begin{equation}
    		x^k \warrow x^* \in C = \bigcap_{i=1}^m \fix{T_i}.
    	\end{equation} 
    \end{theorem} 
    \begin{proof}
    	The given hypotheses make the assumptions of Lemma \ref{lemma: xi-convergence} hold, taking $\mu = 1$.
    	With these, Lemma \ref{lemma: xk-z-convergence} then
    	asserts that $\|x^{k+1}-x^k\|\rarrow 0$
    	and $\{ \|x^k-z\|\}_{k\in \N}$ converges for any $z \in C$.
    	The fact $\|x^{k+1}-x^k\|\rarrow 0$ enables Lemma \ref{lemma: T-operator-residual-convergence} to be applied to deduce $\|T_i(x^k)-x^k\|\rarrow 0$, for each $i\in[m]$.
    	This shows the assumptions of Proposition \ref{proposition: weak-convergence} hold and completes the proof. 
    	\qed
    \end{proof}

    \section{Application to Linear Systems}  
    \label{sec: linear-systems}
    In this section, we present an application for fast solution of linear systems of equations.
    Chiefly, we prove that the operators used in the method of Diagonally Relaxed Orthogonal Projections (DROP) \cite{2008-Censor-DROP} are nonexpansive and, thus, can be incorporated into the ASI algorithm.  
    Let $A \in \R^{M\times N}$ and $b\in\R^M$, and consider the problem
    \begin{equation}
    \mbox{Find $x \in \R^N$ such that $Ax=b$.}
    \label{eq: linear-system}
    \end{equation} 
    We assume that each row and column of the matrix $A$ is nonzero, that $A$ is large, and that the system is overdetermined.
    We provide background material pertaining to projection methods and then show that the ASI algorithm can be used to form an asynchronous version of DROP, which we call ASI--DROP. \\
    
    Each equation in the linear system can be associated with a closed and convex subset of $\R^N$, namely, the hyperplane 
    \begin{equation}
    H_i := \{ x \in \R^N \mid \braket{a^i,x} = b_i \},
    \ \ \mbox{for all $i \in [M]$,}
    \end{equation}
    where $a^i$ is the $i$-th row of $A$.
    The projection $P_i$ onto $H_i$ is given by 
    \begin{equation}
	    P_i(x) = x + \dfrac{b_i - \braket{a^i,x}}{\|a^i\|^2} a^i.
	    \label{eq: projection-on-hyperplane}
    \end{equation}
    In order to discuss blocks of rows from the matrix $A$, let $\{B_t\}_{t=1}^r$ be a collection of sets of indices such that $B_t \subseteq [M]$, for all  $t$, and 
    \begin{equation}
    [M]
    = \bigcup_{t=1}^r B_t.
    \label{eq: block-indices-requirement}
    \end{equation}
    Note that there may be overlapping $B_t$'s, i.e., it is possible that there exists $t\neq s$ such that $B_t\cap B_s \neq \emptyset$.
    
    \subsection{Diagonally Relaxed Orthogonal Projections}           
    The DROP algorithm is a modification of Cimmino's simultaneous projections method \cite{1938-Cimmino}, which uses the iterative process
    \begin{equation}
    	x^{k+1} := x^k - \dfrac{1}{M}\sum_{i=1}^M \rev{\dfrac{\braket{a^i,x}-b_i}{\|a^i\|^2} a^i.}
    	\label{eq: cimmino-iteration}
    \end{equation}
    When $M$ is large, the term $1/M$ restricts the progress of the iteration. 
    Letting $s_j $ be the number of nonzero entries in the $j$-th column of $A$,
    the DROP algorithm replaces the $1/M$ term in  (\ref{eq: cimmino-iteration}) with $1/s_j$     
    to define component-wise updates  via
    \begin{equation}
    	x^{k+1}_j := x^k_j - \dfrac{1}{s_j} \sum_{i=1}^M \rev{\dfrac{\braket{a^i,x}-b_i}{\|a^i\|^2} a^i_j,}
    	\ \ \mbox{for all $j\in [N]$}.
    	\label{eq: DROP-component-wise-update}
    \end{equation}
    \rev{This update (\ref{eq: DROP-component-wise-update}) first appeared in the CAV paper \cite{2001-CAV}, then in the BICAV paper \cite{2001-BICAV}, and also in the paper  that analyzed the convergence of CAV and BICAV \cite{2003-Block-It-Algorithms-Diagonal-Oblique-Projs}. Note that no proof of convergence was provided for (\ref{eq: DROP-component-wise-update}) until the DROP paper \cite{2008-Censor-DROP} and the CARP paper \cite{2005-CARP}, where we note that (\ref{eq: DROP-component-wise-update}) is a special case of CARP when each block consists of a single equation, which is called CARP1.}      
    If $A$ is sparse, then $s_j\ll M$ and this approach effectively makes the update $x^{k+1}_j$ depend upon $x_j^k$ and the average over the summands for which $a_j^i$ is nonzero. 
    \rev{For each $t \in [r]$, define $A_t$ and $b_t$ to be the submatrix and subvector of $A$ and $b$, respectively, corresponding to the row indices in $B_t$, and define the matrices 
    \begin{equation}
	    W_t := \mbox{diag}\left( \|a^i\|^{-2} \right)  ,
	    \ \ \
	    D_t := \mbox{diag}\left( 1/s_j \right) ,
	    \ \ \ 
	    \overline{A}_t := A_tD_t^{1/2}  ,
   		\label{eq: DROP-matrices-def} 
    \end{equation}
    where the $i$-th diagonal entry of $W_t\in \R^{M\times M}$ is the inverse of the square of the norm of the $i$-th row in $A_t$ and the $j$-th diagonal entry in $D_t \in \R^{N\times N}$ gives the number of nonzero entries in the $j$-th column of $A_t$.}
   For each $t\in [r]$ define the operator $U_t:\R^N\rightarrow\R^N$ by 
    \begin{equation}
	    U_t(y) := y - \overline{A}_t^T W_t\left(\overline{A}_ty - b_t\right),
	    \label{eq: DROP-block-operator}
    \end{equation}	
    where $b_t$ is the subvector of $b$ corresponding to the row indices in $B_t$ and here $T$ stands for matrix transposition.
    Given an almost cyclic control $\{t_k\}_{k\in\N}$ on $[r]$, a special case of the DROP algorithm is given by the process
    \begin{equation}
    	y^{k+1} := U_{t_k}\left( y^k \right).
    \end{equation}
    The following proposition shows this formulation of DROP can be used with the ASI algorithm.
	\begin{proposition}
		\label{prop: DROP-NE}
		Each operator $U_t$ defined by (\ref{eq: DROP-block-operator})  is nonexpansive with respect to the Euclidean norm.
	\end{proposition}
	\begin{proof}		 
		Let $t \in [r]$ and let $(\mu,v)$ be an eigenpair of $\overline{A}_t^T W_t\overline{A}_t$ so that $\mu v = \overline{A}_t^T W_t\overline{A}_t v$. Multiplying by the transpose of $v$ and dividing by $\braket{v,v}$ yields
		\begin{equation}
		\mu = \dfrac{\braket{v,\overline{A}_t^T W_t\overline{A}_t v} }{\braket{v,v}}
		= \dfrac{\|\overline{A}_t v\|_{W_t}^2}{\|v\|_2^2}
		\geq 0,
		\end{equation}
		where $\|v\|_2^2 = \braket{v,v}$ with $\braket{\cdot,\cdot}$ the inner product in $\R^N$ and $\|\cdot \|_{W_t}$ is the $W_t$-norm defined by $\|y\|_{W_t}^2 = \braket{y,W_ty}$. Note that $W_t$ is symmetric and positive definite; thus, this is well-defined.
		Since the matrix $\overline{A}_t^T W_t\overline{A}_t$ is symmetric, its eigenvalues are real. 
		Additionally,  $\overline{A}_t^T W_t\overline{A}_t$ and $D_tA_t^T W_tA_t$ have the same eigenvalues.
		Indeed, using $\det$ for determinants and $\mbox{Id}$ for the identity matrix,  we use properties of determinants to see that
		\begin{equation} 
		\begin{aligned}
		\det\left( \mu \mbox{Id} - D_tA_t^T W_tA_t \right)
		&= \det(D_t^{-1/2})\det\left( \mu \mbox{Id} - D_tA_t^T W_tA_t \right) \det(D_t^{1/2})\\
		&= \det\left( \mu \mbox{Id} - D_t^{1/2} A_t^T W_tA_tD_t^{1/2}  \right)\\
		&= \det \left(\mu \mbox{Id} - \overline{A}_t^T W_t\overline{A}_t\right).
		\end{aligned}
		\end{equation}				
		By \cite[Lemma 2.2]{2008-Censor-DROP}, we know that $\rho(D_tA_t^TW_tA_t) \leq 1.$
		Consequently,  $\mu \in [0,1]$.
		Note that
		\begin{equation}
		(\mbox{Id}-\overline{A}_t^T W_t\overline{A}_t) v
		= (1-\mu) v,
		\end{equation}
		and $(1-\mu) \in [0,1]$, which implies
		$\rho (\mbox{Id}-\overline{A}_t^T W_t\overline{A}_t) \leq 1$.
		Moreover, because $(\mbox{Id}-\overline{A}_t^T W_t\overline{A}_t)$ is symmetric,
		\begin{equation}
		\|\mbox{Id}-\overline{A}_t^T W_t\overline{A}_t\|_2
		= \rho(\mbox{Id}-\overline{A}_t^T W_t\overline{A}_t)
		\leq 1.
		\end{equation}
		Whence, for any $y^1,y^2 \in \R^N$
		\begin{equation}
		\begin{aligned}
		\|U_t(y^1) - U_t(y^2) \|_2
		&= \left\| \left( \mbox{Id} - \overline{A}_t^T W_t \overline{A}_t\right) \left(y^1-y^2 \right) \right\|_2\\
		&\leq \|\mbox{Id}-\overline{A}_t^T W_t\overline{A}_t\|_2 \|y^1 -y^2\|_2\\
		&\leq \|y^1 - y^2\|_2,
		\end{aligned}
		\end{equation}
		and we conclude that $U_t$ is nonexpansive.
		\qed
	\end{proof}

    The ASI algorithm with these DROP operators, henceforth called the ASI--DROP algorithm, takes the following form.
    Its presentation is followed by a theorem guaranteeing its convergence.
    
    \begin{algorithm}[h]
    	\caption{ASI--DROP \label{alg: ASI--DROP}}
    	\normalsize
    	\hspace*{0 pt}
    	{\bf Initialization:}
    	Let $A \in \R^{M\times N}$ and $b \in \R^M$ be given.
    	Choose any $x^1 \in \R^N$, a sequence $\{\lambda_k\}_{k\in\N}$ such that $\lambda_k \in (0,1)$ for all $k\in\N$, an almost cyclic control $\{t_k\}_{k\in\N}$ on $[r]$, and  a family of blocks of indices $\{B_t\}_{t=1}^r$  satisfying (\ref{eq: block-indices-requirement}).\\[10 pt]
    	 
    	\hspace*{0 pt}
    	{\bf Iteration:}
    	For each $k\in \N$ set    	
    	\begin{equation}
    	x^{k+1} :=
    	\begin{cases}
    		\begin{array}{ll}
    				x^k, 
    				& \mbox{\ if $k\leq \sup_{k\in\N} \|d^k\|_\infty$,}\\
    		    	x^k - \lambda_k D_{t_k} A_{t_k}^T W_{t_k} \left( A_{t_k} \hat{x}^k - b_{t_k} \right), 
    		    	& \mbox{\ otherwise.}
    		\end{array}
    	\end{cases}
    	\end{equation}
    \end{algorithm}     
    
    \begin{theorem}
    	Let $\{x^k\}_{k\in\N}$ be a sequence generated by the ASI--DROP algorithm.
	   	If the linear system (\ref{eq: linear-system}) is consistent, and if the delay vectors  are uniformly bounded in the sup norm by some $\tau \geq 0$, and if there is $\varepsilon > 0$ such that
	   	\begin{equation}
	   	0 < \varepsilon \leq \lambda_k \leq \dfrac{1}{2\tau + 1 + \varepsilon},
	   	\ \ \mbox{for all $k\in\N$,}
	   	\end{equation}
	   	then the sequence $\{x^k\}_{k\in\N}$ converges to a solution of the linear system (\ref{eq: linear-system}).
    \end{theorem}
    
    \begin{proof} 
	    For each $t$, set $S_t := \mbox{Id} - U_t$.
	    Proposition \ref{prop: DROP-NE} and Theorem \ref{theorem: main-result} imply that the sequence $\{y^k\}_{k\in\N}$ generated by the ASI algorithm converge to a fixed point of $U_t$.

	  With the above notations and with $U_t$ as in (\ref{eq: DROP-block-operator}), define for each $t$
	  \begin{equation}
	  S_t := \mbox{Id} - U_t.
	  \end{equation}
	    Let $y^* = \limk y^k$.
	    Then for each $t\in [r]$
	    \begin{equation}
	    b_t = \overline{A}_{t} y^* 
	    = A_t D^{1/2}_t y^*
	    = A_t x^*,
	    \end{equation}
	    taking $x^* := D^{1/2} y^*$.
	    This implies that $b = AD^{1/2} y^* = Ax^*$.	
	    For each $k\in \N$ set $x^k = D_{t_k}^{1/2} y^k$.
	    Then $x^k \rarrow x^*$ where $x^*$ is a solution to the linear system (\ref{eq: linear-system}). Moreover, the active step describing the iterate updates is
	    \begin{equation}
	    	\begin{aligned}
	    		x^{k+1}
	    		& = D_{t_k}^{1/2} y^{k+1} \\
	    		& = D_{t_k}^{1/2}\left(y^k - \lambda_k S_{t_k}\left(\hat{x}^k\right) \right) \\
	  			&=  D_{t_k}^{1/2}\left(y^k - \lambda_k (I-U_{t_k}) \left(\hat{y}^k\right)\right) \\
	  			&= D_{t_k}^{1/2} \left( y^k - \lambda_k \overline{A}_{t_k}^T W_{t_k} \left( \overline{A}_{t_k} \hat{y}^k - b_{t_k} \right)\right) \\
	  			& = x^k - \lambda_k D_{t_k} A^T_{t_k} W_{t_k}\left(A_{t_k} \hat{x}^k - b_{t_k} \right).
	    	\end{aligned}
	    \end{equation}
	    This completes the proof. 
	    \qed
    \end{proof}

    \subsection{Kaczmarz's and Other Methods}
    
    {\bf Kaczmarz Method.} A popular method for approximating solutions to linear systems is that of Kaczmarz \cite{1937-Kaczmarz}. In the context of computed tomography, this method is known as the Algebraic Reconstruction Technique (ART) since it was rediscovered there by Gordon, Bender, and Herman in \cite{1970-ART}.
    Kaczmarz's method generates updates by successively projecting each iterate $x^k$ onto an individual hyperplane $H_i$. For an almost cyclic control $\{i_k\}_{k\in\N}$ on $[M]$ and a sequence of scalars $\{\lambda_k\}_{k\in\N}$, updates in the relaxed version of Kaczmarz's method are given by the iteration
    \begin{equation}
    	x^{k+1}
    	= (1-\lambda_k) x^k + \lambda_k P_{i_k}\left(x^k\right)
    	=  x^k + \lambda_k  \left(\dfrac{b_{i_k}-\braket{a^{i_k},x^k}}{\|a^{i_k}\|^2}\right)a^{i_k},
    \end{equation}
    where $P_{i_k}$ is as in (\ref{eq: projection-on-hyperplane}).
    Since the projection operators $\{P_i\}_{i=1}^M$ are nonexpansive, we use them in the ASI framework to construct an asynchronous generalization.
    For each $i \in [M]$ set, analogously to (\ref{eq: Original-Si-def}),
    \begin{equation}
    	S_i(x) := (\I- P_i)(x)
    	= \dfrac{\braket{a^i,x}-b_i}{\|a^i\|^2} a^i.
    \end{equation}
    Then the ASI--ART method is presented formally in Algorithm \ref{alg: ASI-ART}. \\

    \begin{algorithm}[H]
    	\caption{ASI-ART \label{alg: ASI-ART}}
    	\normalsize
    	\hspace*{0 pt}
    	{\bf Initialization:}
    	Let $A \in \R^{M\times N}$ and $b \in \R^M$ be given.
    	Choose any $x^1 \in \R^N$,  a sequence $\{\lambda_k\}_{k\in\N}$ such that $\lambda_k \in (0,1)$ for all $k\in\N$, and an almost cyclic control $\{i_k\}_{k\in\N}$ on $[M]$.\\[10 pt]
    	
    	\hspace*{0 pt}
    	{\bf Iteration:} 
    	For each $k\in \N$ set
    	\begin{equation}
    	x^{k+1} :=
	    	\begin{cases}
	    	\begin{array}{ll}
	    	x^k, 
	    	& \mbox{\ if $k\leq \sup_{k\in\N} \|d^k\|_\infty$,}\\
	    	x^k - \lambda_k \dfrac{\braket{a^{i_k},\hat{x}^k}-b_{i_k}}{\|a^{i_k}\|^2} a^{i_k}, 
	    	& \mbox{\ otherwise.}
	    	\end{array}
	    	\end{cases}
    	\end{equation}    	
    \end{algorithm}

    {\bf Other Methods.} The ASI framework can be applied in conjunction with other projection methods. These include the valiant projection method (VPM) of Censor and Mansour \cite{2018-Censor-VPM}, which is known as the automatic relaxation method (ARM) of Censor \cite{1985-Censor-ARM} when applied to interval linear inequalities. We conjecture the intrepid method of Bauschke, Iorio, and Koch \cite{2014-Bauschke-Intrepid}, which is known as the ART3 method of Herman \cite{1975-Herman-ART3} when applied to linear systems, may also be used within the ASI framework.

   \section{ASI Algorithm Implementation}
	\label{sec: ASI-implementation} 		
		A sample pseudocode of the ASI algorithm is presented formally in Algorithm \ref{alg: ASI-implementation} and illustrated in Figure \ref{fig: ASI-illustration}. Our model is for a master/slave type architecture.
		Referring to the notation of Section \ref{sec: ASI-alg},
		let $x \in \H$ be arbitrary, fix $\lambda \in (0,1)$, and let $\{i_k\}_{k\in\N}$ be an almost cyclic control on $[m]$.
		First the initial iterate $x$ is sent to each of the $w$ processing nodes and we let the $\ell$-th node compute $N_\ell := S_{i_\ell}(x)$.
		The iteration counter $k$  is then set to $w+1$ and the time step counter $\theta$ to $1$. Note the time step $\theta$ is distinct from the iteration step $k$ since multiple nodes may complete their updates at the same time step, thereby enabling the iteration step to exceed the time step (i.e., $k\geq \theta$).
		The iterative process proceeds by fetching the collection of indices $F_\theta$ of nodes that produce outputs at time $\theta$. Then, for each $\ell \in F_\theta$, we update $x$ with $x - \lambda N_\ell$, where $N_\ell$ is the output of the $\ell$-th node. Then  $x$ is sent to the $\ell$-th node to compute $N_\ell = S_{i_k} (x)$.
		After the loop occurs over all elements of $F_\theta$, we increment $\theta$ by 1 and repeat the iteration step if the stopping criteria are not met. \\

	\begin{algorithm}[t]
		\caption{A Pseudocode Implementation of the ASI Algorithm\label{alg: ASI-implementation}}
		\normalsize
		{\bf Initialization:} \\
		\tab Let $x \in \H$, $\lambda \in (0,1)$, and $\{i_k\}_{k\in\N}$ an almost cyclic control on $[m]$. \\ 
		\tab {\bf for} $\ell \in [w]$ \\
		\tab \tab Send $x$ and $i_\ell$ to the $\ell$-th node to compute $N_\ell = S_{i_\ell}(x)$   \\
		\tab {\bf endfor} \\
		\tab $k\leftarrow w+1$ \\
		\tab $\theta\leftarrow 1$ \\ \ \\
		
		{\bf Master Node Iteration:} \\
		\tab {\bf while} stopping criteria not met \\
		\tab \tab Fetch set of node indices $F_\theta$ for outputs received at time $\theta$\\
		\tab \tab {\bf for} $\ell \in F_\theta$ \\
		\tab \tab \tab $x\leftarrow x - \lambda N_\ell $ \\
		\tab \tab \tab $k\leftarrow k + 1 $ \\
		\tab \tab \tab  Send $x$ and $i_k$ to $\ell$-th slave node  to compute	$N_\ell = S_{i_k}(x)$ \\
		\tab \tab {\bf endfor} \\		
		\tab \tab $\theta\leftarrow \theta+ 1$\\
		\tab {\bf end while}	\\ \ \\
		
		{\bf Slave Node $\ell$ Iteration:} \\
		\tab Read $x$ and $i_k$ as input \\
		\tab Compute $N_\ell = S_{i_k}(x)$\\
		\tab Output $N_\ell = S_{i_k}(x)$ to master node
	\end{algorithm} 	
	
	\begin{figure}[t] 
		\centering 
		{\large
			\begin{tikzpicture}[square/.style={regular polygon,regular polygon sides=4}]
			\def\P{1.}
			\coordinate (c1) at (0,0);
			\coordinate (c2) at (0,-2*\P);
			\coordinate (c3) at (0,-4*\P);
			\coordinate (c4) at (2.5*\P, -1*\P);
			\coordinate (c5) at (2.5*\P, -3*\P);
			\coordinate (c6) at (4.75*\P, -2*\P);

			\begin{scope}[very thick, every node/.style={sloped,allow upside down}]

			
			\draw[]  ($(c2)$)  
			-- node {\midarrow} ($(c4)+(-0.35*\P,0)$) ;
			

			\coordinate (p1) at ($(c6)+(0.7*\P,0)$);
			\coordinate (p2) at ($(c6)+(1.5*\P,0)$);
			\coordinate (p3) at ($(c6)+(1.5*\P,-3*\P)$);
			\coordinate (p4) at ($(c6)+(-5.75*\P,-3*\P)$);
			
			\coordinate (p5) at ($(-1.0*\P,-4*\P)$);
			\coordinate (p6) at ($(-0.25*\P,-4*\P)$);
			
			\coordinate (p7) at ($(-1.0*\P,-2*\P)$);
			\coordinate (p8) at ($(-0.25*\P,-2*\P)$);
			
			\coordinate (p9)  at ($(-1.0*\P,0*\P)$);
			\coordinate (p10) at ($(-0.25*\P,0*\P)$);
			
			\coordinate (q1) at ($(c5)!0.1!(c6)$);
			\coordinate (q2) at ($(c5)!0.9!(c6)$);
			\coordinate (q3) at ($(c4)!0.2!(c6)$);
			\coordinate (q4) at ($(c4)!0.9!(c6)$);	
			\coordinate (q5) at ($(c5)+(0,-2*\P)$);			
			\coordinate (q6) at ($(c5)+(0,-0.5*\P)$);	
			
			\draw[] (q1) -- node {\midarrow} (q2);
			\draw[] (q3) -- node {\midarrow} (q4);
			\draw[] (q5) -- node {\midarrow} (q6);
			
			\foreach \n in {1,2,3} {%
				\pgfmathtruncatemacro{\m}{\n+1}%
				\draw[] (p\n)  -- node {\midarrow} (p\m);
			}
			
			\draw[] (p4) -- node {\midarrow} (p5);
			
			\draw[] (p5) -- node {\midarrow} (p7);
			\draw[] (p7) -- node {\midarrow} (p8);
			
			
			\draw[fill = black!10!white] (c4) circle (0.85) node[] {$S_{i_k}(\hat{x}^k)$};	
			
			\draw[draw, fill=gray!20!] (c5) circle (0.65) node[] {${x}^k$};		
			
			\draw[draw, fill=gray!20!] (c6) circle (0.65) node[] {${x}^{k+1}$};			
			\end{scope}

			\foreach \x in {1,2,3} {%
				\draw[fill = gray!50!] (c\x) circle (0.5*\P);
			}%
			
			\draw[] (c1) node[] {$N_1$};
			\draw[] ($(c1)!0.5!(c2)$) node[] {{\Huge$\vdots$}};
			\draw[] (c2) node[] {$N_\ell$};
			\draw[] ($(c2)!0.5!(c3)$) node[] {{\Huge$\vdots$}};
			\draw[] (c3) node[] {$N_w$};		
			
			\end{tikzpicture}}
		\caption{Schematic architecture for the ASI algorithm. At the current iteration $k$ the latest output $N_\ell$\rev{, which is precisely $S_{i_k}(\hat{x}^k)$ in the ASI algorithm,} from the $\ell$-th node is merged with $x^k$ via a linear combination to form the update $x^{k+1}$, overwriting \rev{the global variable} $x^k$. Here $w\leq m$ is the number of processing nodes.}
		\label{fig: ASI-illustration}
	\end{figure}
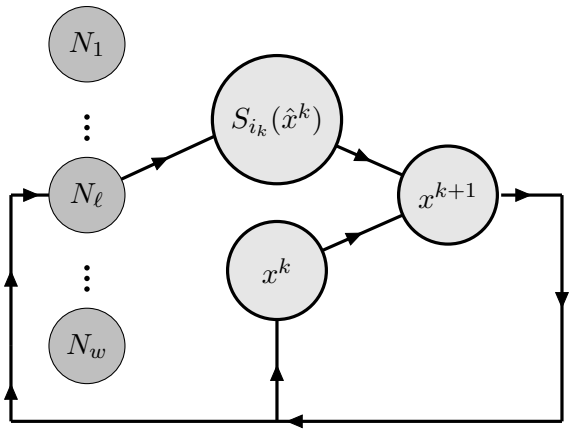		
		
		In the schematic Figure \ref{fig: ASI-illustration}, each processing node is represented by a circle with $N_\ell$ inside. At iteration step $k$, the most recent output from the collection of nodes is fetched, which is precisely $N_\ell = S_{i_k}(\hat{x}^k)$ when the most recent output is from the $\ell$-th node. 
		This is then merged with $x^k$ as in the ASI algorithm to form the new iterate $x^{k+1}$, overwriting $x^k$. The output $x^{k+1}$ is then fed to the $\ell$-th node to compute $N_\ell = S_{i_{k+1}}(x^{k+1})$.
		This effectively sets  $k\leftarrow k+1$. Then the process repeats, fetching the most recent node outputs. In this master/slave framework, each of the $w$ slave nodes applies operators from the family $\{S_i\}_{i=1}^m$ and the master node continually computes the updates by merging $x^k$ with $S_{i_k}(\hat{x}^k)$. 
		
		\begin{remark}
			Implementation of Algorithm \ref{alg: ASI--DROP} can be done using the pseudocode in Algorithm \ref{alg: ASI-implementation} by taking $\H = \R^N$, $i_k = t_k$ for all $k\in \N$, $m = r$, and $\{U_t\}_{t=1}^r = \{T_i\}_{i=1}^m$, the family of DROP operators associated with the family of blocks of indices $\{B_t\}_{t=1}^r $ in Algorithm \ref{alg: ASI--DROP}.
			Similarly, Algorithm \ref{alg: ASI-ART} can be implemented using the pseudocode in Algorithm \ref{alg: ASI-implementation} by taking $\H = \R^N$ and $m = M$.
		\end{remark} 
		
		\begin{remark}
			\label{remark: node-i_k}
			The indexing of the slave nodes can be set up as follows.
			First locally store subcollections of the family of operators $\{S_i\}_{i=1}^m$ to each slave node.
			Then let each node progress cyclically through the operators it has stored locally.
			Note that, although each slave node may proceed in a cyclic fashion, the order in which outputs arrive to the master node may not be cyclic.
			This can result from the variation in computation times on each node and, thus, arrival times to the master node.
			Despite this, the arrival of outputs to the master node will still occur in an almost cyclic fashion, as needed.
		\end{remark}

   \section{A Computational Example}
   \label{sec: example}

   The ASI algorithm is written in a sequential manner; however, we  note that being able to use out-of-date iterations enables all the processing nodes to work {\it simultaneously}. The processing nodes are also able to work independently of each other. 
   Below we apply our results in a computational example with a model computed tomography (CT) image reconstruction problem in Matlab. We use Algorithm \ref{alg: ASI-implementation} to implement Algorithm \ref{alg: ASI--DROP} and take $r = m = 40$ and $t_k = i_k,$ for all $k\in\N$.

	\subsection{Experiment Setup}

    In our computational example, we provide results using the method of
    Diagonally Relaxed Orthogonal Projections (DROP) \cite{2008-Censor-DROP}.
    We implement DROP using both the ASI algorithm and the form of asynchrony of  \cite{1992-Elsner}, i.e., without the inertial terms. 
    We refer to these as ASI--DROP and EKN--DROP, respectively (EKN is for the authors' names of \cite{1992-Elsner}).
    Note that convergence for EKN--DROP is not proven since we have not validated that the DROP operators are paracontractions. However, as expected, this method converged in all our experiments. \\
        
	 \begin{figure}[t]
	 	\centering
	 	\includegraphics[width = 2 in]{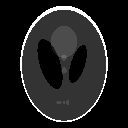}
	 	\caption{A 128x128 digitization of the Shepp-Logan phantom \cite{1974-Shepp-Logan}.}
	 	\label{fig: Shepp-Logan}
	 \end{figure}          

	We consider an image reconstruction model of CT image reconstruction. 	
	Our specific aim is to show how an inherently sequential algorithm can be executed by multiple nodes working in parallel and asynchronously.    	
	This example illustrates the speedup of DROP for solving linear systems and the speedup that occurs when using the ASI--DROP algorithm.
	The task at hand is to solve (\ref{eq: linear-system}) where $A$ is a given 176,672$\times$16,384 matrix and $b$ is a vector with 176,672 entries. The computational work was done in Matlab and
	the quantities $A$ and $b$ were generated using the Shepp-Logan phantom \cite{1974-Shepp-Logan} in Figure \ref{fig: Shepp-Logan} and the AIR Tools Matlab package \cite{2012-Hansen-AIR-Tools}.
	In our implementations, we generated $\{i_k\}_{k\in\N}$ following Remark \ref{remark: node-i_k}, i.e., the family of operators $\{S_i\}_{i=1}^m$ is loaded into memory and each slave node accesses a subcollection of this family and applies the operators from that subcollection cyclically.	
	Execution was stopped when sufficient proximity was reached. In particular, we stopped the iterations when $\|x^k - x\| < \varepsilon = \expepsilon$, where $x$ is the true image vector, i.e., the ``phantom'' from which $A$ and $b$ were reconstructed for the purpose of the reconstruction experiment.
	We define an epoch to be the number of operators $m$, which in the case of our experiment was $m=40$.
	The computation cluster used had 49.5 GB of RAM and 12 Intel Xeon X5650 processors with frequency 2.67 GHz. 
	For more in-depth material on CT image reconstruction, we refer the reader, e.g., to Herman's book \cite{2009-Herman-Book}. 
	For the asynchronous calls to each processing node, we used the \verb|parfeval| command in Matlab's Parallel Computing toolbox.\\

	 \begin{table}[t]
	 	\centering
	 	\begin{tabular}{|c|c|c|c|c|c|c|}	 			
	 		\hline
	 		\multirow{2}{*}{Method} & \multirow{2}{*}{Measurement} & \multicolumn{5}{|c|}{number of slave nodes $(w)$} \\\cline{3-7}
	 		&   		
	 		& $w=1$ 
	 		& $w=2$
	 		& $w=4$ 
	 		& $w=8$ 
	 		& $w=10$ 
	 		\\\hline	 		  		 		
	 		\multirow{3}{*}{ASI--DROP} 
	 		& time (sec)   & 751.5 & 418.0 & 226.9 & 140.3 & 163.8 \\\cline{2-7}
	 		& \# epochs & 353.9 & 357.4 & 352.5 & 352.4 & 445.0  \\\cline{2-7}
	 		& speedup & NA & 1.80 & 3.31 & 5.35 & 4.59 \\\hline
	 		\multirow{2}{*}{EKN--DROP} &
	 		time (sec)& 767.1 & 540.8 & 387.9 & 361.1 & 395.6  \\\cline{2-7}
	 		& \# epochs   & 353.9 & 427.5 & 561.4 & 840.7 & 989.0 \\\cline{2-7}
	 		& speedup & NA & 1.42 & 1.98 & 2.12 & 1.94 \\\hline			
	 	\end{tabular}
	 	
	 	\caption{Reconstruction results with iterations stopped when $\|x^k-x\| < \varepsilon = \expepsilon$. Reported values are averaged from \numtrials trials repeated on the same data set.}
	 	\label{table: lambda-0.20}
	 \end{table}       
    \subsection{Numerical Results}
    
    Reconstruction results are provided in Table \ref{table: lambda-0.20}.
    Note that the number of iterations required for the EKN--DROP approach increases as the number of nodes $w$ increases. However, in some cases, {\it fewer} iterations are required using the ASI--DROP algorithm. Furthermore, there is speedup as the number of nodes increases, given by 1.80, 3.31, 5.35, and 4.59 for 2, 4, 8, and 10 nodes, respectively. With $w=12$ nodes, the ASI algorithm does not converge with the step size $\lambda=0.20$, but does converge if the step-size $\lambda$ is sufficiently reduced (e.g., by taking $\lambda = 0.15$).
    In comparison, for EKN--DROP the speedup was 1.42, 1.98, 2.12, 1.94 for 2, 4, 8, and 10 nodes, respectively.    
    The fastest reconstruction time for ASI--DROP was 140.3 seconds, which is more than twice as fast as that for EKN--DROP (361.1 seconds) and over five times faster than the sequential implementation of DROP.  \\
    
    For both forms of asynchrony, we see diminishing returns as $w$ increases.
    It might even be more advantageous to choose $w$   smaller rather than larger.
    This is seen by comparing the ASI--DROP algorithm performance when $w=8$ and $w=10$ in Table \ref{table: lambda-0.20}.
    However, for sufficiently small $w > 1$ we still see notable performance improvements over the existing synchronous sequential case $w=1$.
    Plots of the averages of residuals for the ASI--DROP and EKN--DROP schemes are provided in Figures \ref{fig: residual-plots-comparison},  \ref{fig: residual-plots}, and \ref{fig: residual-plots-div}.    
    Figure \ref{fig: residual-plots}(a) reveals that, when the step sizes are small enough, roughly the same number of iterations is needed to obtain convergence of the ASI--DROP algorithm as $w$ increases, up to $w=8$. Figure  \ref{fig: residual-plots}(b) shows that the amount of computation time decreases as $w$ increases, up to $w=8$, for the ASI--DROP algorithm.
    For the EKN--DROP algorithm, Figure  \ref{fig: residual-plots}(c) shows more iterations are needed to obtain convergence; however, even with the larger number of iterations, Figure  \ref{fig: residual-plots}(d) shows speedup is still obtained for $w>1$.
    Figure \ref{fig: residual-plots-div} demonstrates the behavior of the ASI--DROP algorithm when the step sizes increase to nearly the maximal size that still yields convergence. There we see a wide plot, which demonstrates the variance among arrival times of the iterates generated by the ASI--DROP algorithm. Note also that the iterates do {\it not} necessarily approach the solution monotonically.
    In Figure \ref{fig: reconstructions}, we see sample reconstructions that reveal each image high equality. 
    In Figure \ref{fig: residual-plots-comparison}, we see the primary comparison plot between the ASI--DROP and EKN--DROP methods of asynchrony, which shows increasing speedup as $w$ increases for the ASI--DROP algorithm. 
    In summary, compared to the sequential application of block operators, both versions of asynchrony display speedup while the ASI--DROP approach is faster when it converges. 
     \begin{figure}[t] 
     	\centering
     	\includegraphics[width = 3.75 in]{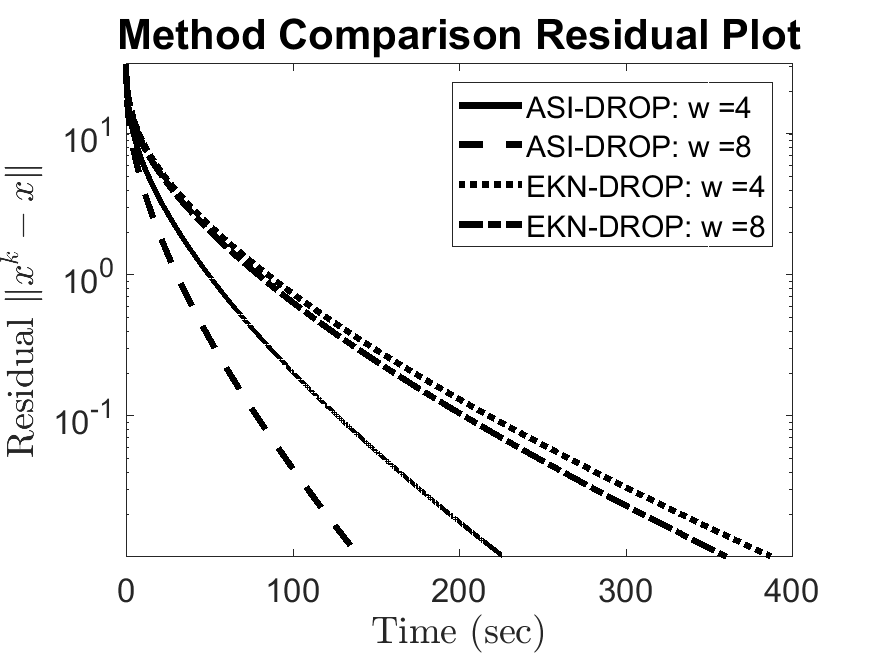}	 
     	\caption{Juxtaposition of averages of residual plots over \numtrials trials for the ASI--DROP and EKN--DROP algorithms.}
     	\label{fig: residual-plots-comparison}
     \end{figure}

    \subsection{Discussion}  
    From the computational example, we learn that much larger step sizes may possibly be used in practice, with promise, than the bound given in Theorem \ref{theorem: main-result}.
    If the updates are uniformly random and independent of the distribution of delays, then the results of ARock \cite[Table 2]{2016-Hannah-Unbounded} can be used to deduce larger step sizes yield convergence, i.e., the upper bound for step sizes is  $1/(1+2\tau/\sqrt{m})$, where $m$ is the number of operators and $\tau$ is an upper bound on the delays.
   	However, once the number of nodes $w$ increases too much, then the step size must be reduced to maintain convergence. There appears to be an optimal pairing $(\lambda,w)$ for obtaining the fastest reconstructions.
   	The EKN asynchrony may be advantageous over a sequential algorithm, but requires an increasing number of iterations as $w$ increases, thereby limiting the speedup. Due to the introduction of inertial terms (c.f. (\ref{eq: update-lin-comb-plus-pert}) and Figure \ref{fig: ASI-illustration}), we see roughly the same number of iterations may be needed for the ASI algorithm as $w$ increases, and in some cases {\it fewer} iterations are needed.\\

    Although there is a speedup for the ASI--DROP algorithm, it is sublinear in our experiments. Initially, this may seem contradictory since fewer iterations are required and more processing nodes are used. However, when several nodes produce outputs during the same time step $\theta$, a queue is formed for the updates to $x^k$ (c.f. the \verb|for| loop during the iteration of Algorithm \ref{alg: ASI-implementation}). This leads to some delays and node idle time.
    If the operators are more computationally expensive, then the chances of simultaneous node outputs at the same time step goes down. Consequently, it may be advantageous to use few costly operators rather than many computationally cheap operators. Alternatively, one may choose a different pseudocode implementation of the ASI algorithm that does not utilize the master/slave architecture, but instead uses some form of peer to peer network.
     
   \section{Conclusion}  \label{sec: conclusion}
   In this work, we present a KM-type iteration for common fixed point problems that allows for partial asynchronity, i.e., delays uniformly bounded in the sup norm by some $\tau \geq 0$.
   Convergence of this ASI algorithm is established.  This provides robustness to dropped network transmissions, removes both the need to  synchronize node outputs and to coordinate load-balancing, and reveals a method for attaining further speedup with block-iterative methods when solving large scale problems.  
   Moreover, when there is a delay, inertial terms are introduced into the iteration to accelerate convergence.
   In some cases, this reduces the number of iterations needed to converge.   
   Future work may test the ASI algorithm on massive scale problems, study application of the ASI algorithm to computing architectures other than the master/slave architecture, and investigate extensions to inconsistent convex feasibility problems.    
     
 	\section*{Acknowledgments}
 	We thank Prof. Reinhard Schulte from Loma Linda Univeristy for valuable discussions and insights regarding the advantage of an asynchronous approach to large scale problems in proton CT (pCT) image reconstruction and Intensity-Modulated Proton Therapy (IMPT).
 	We appreciate constructive comments received from Prof. Ernesto Gomez from California State University San Bernardino and Prof. Keith Schubert from Baylor University.
 	\rev{We thank the anonymous reviewers for their constructive comments. We are indebted to Patrick Combettes, Jonathan Eckstein, Dan Gordon and Simeon Reich for valuable feedback on the earlier version of this work that was posted on arXiv.}
 	The first author's work is supported by the National Science Foundation Graduate Research Fellowship under Grant No. DGE-1650604.
 	Any opinion, findings, and conclusions or recommendations expressed in this material are those of the authors and do not necessarily reflect the views of the National Science Foundation. The second author's work was supported by research grant No. 2013003 of the United States-Israel Binational Science Foundation (BSF).

 \begin{figure}[t]
 	\centering
	 	\subfloat[ASI--DROP]{\includegraphics[width = 2.5 in]{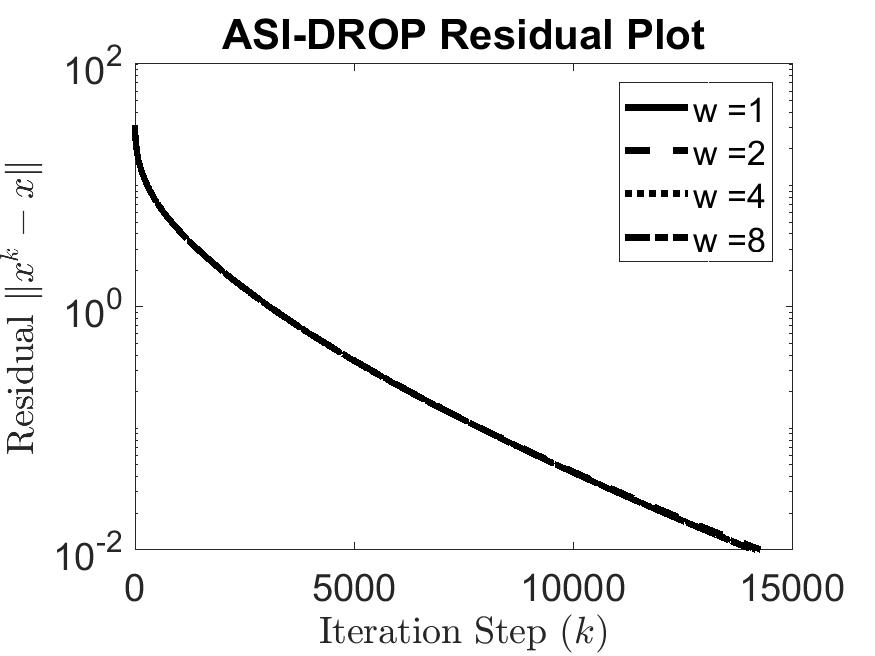}} 	
 	\hspace*{0.0 in} 
 	\subfloat[ASI--DROP]{\includegraphics[width = 2.5 in]{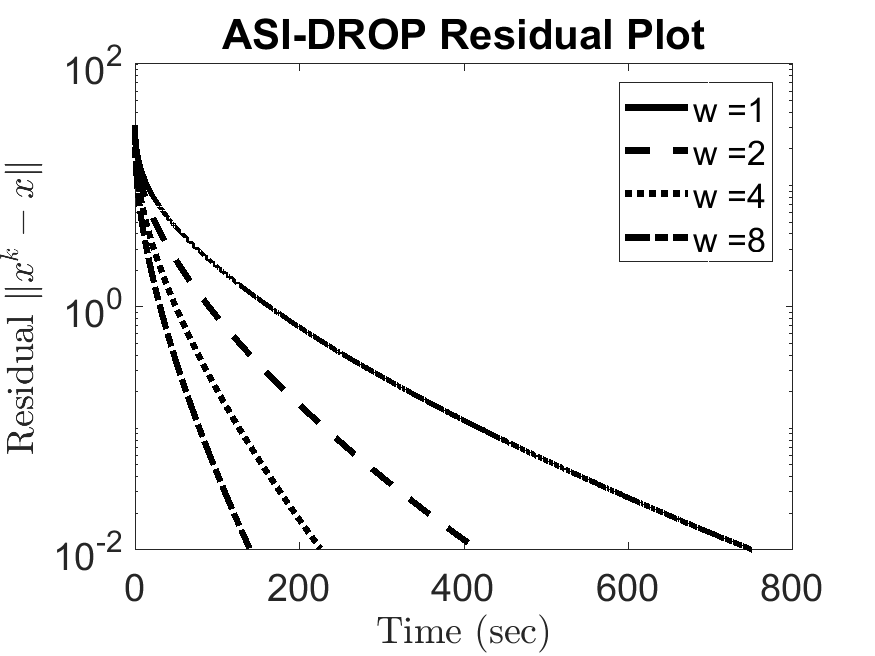}} 	 
 	\ \\	 
 	\subfloat[EKN--DROP]{\includegraphics[width = 2.5 in]{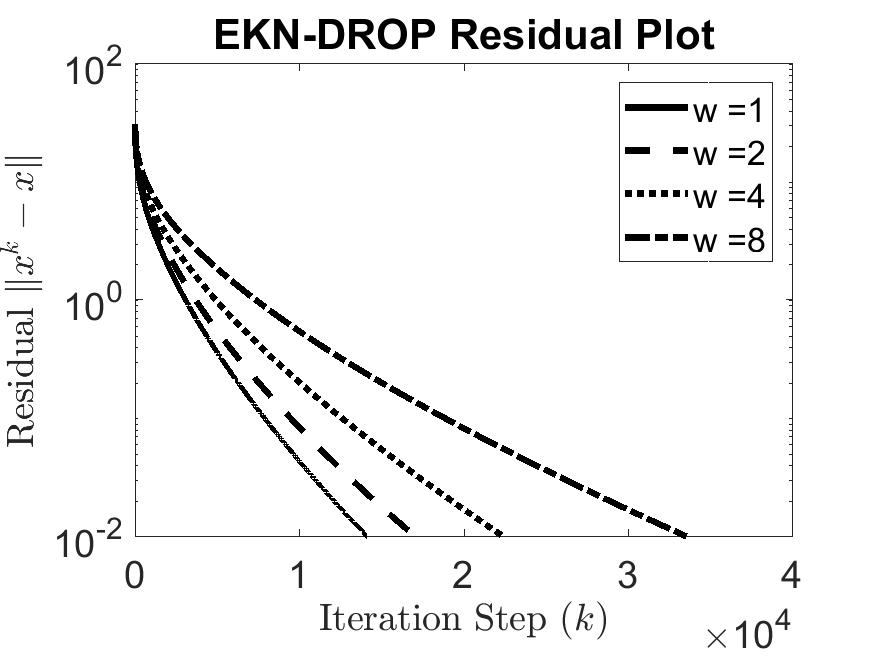}} 	
 	\hspace*{0.0 in} 
 	\subfloat[EKN--DROP]{\includegraphics[width = 2.5 in]{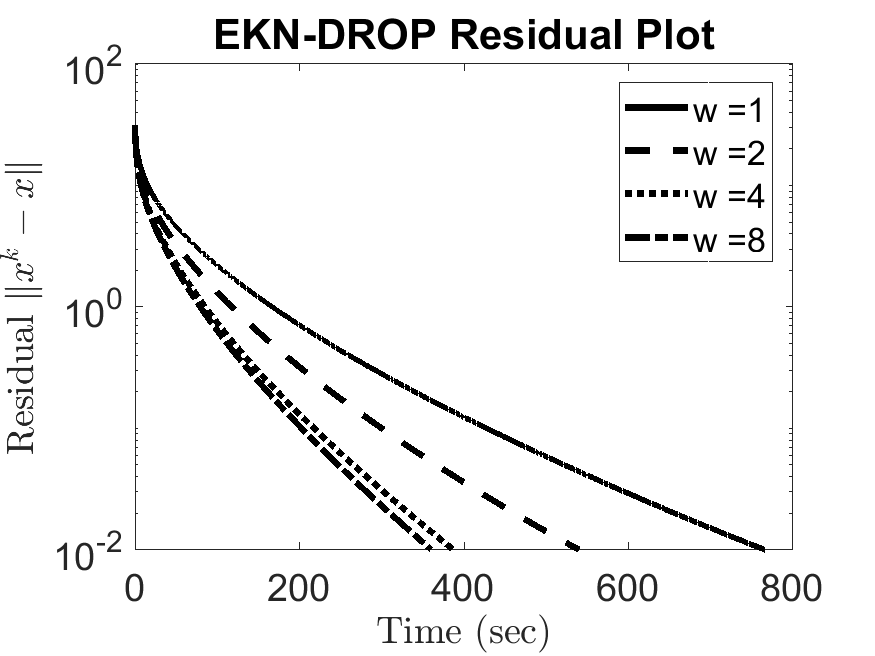}} 	 
 	\caption{Averages of residual plots over \numtrials trials for the ASI--DROP and EKN--DROP algorithms.}
 	\label{fig: residual-plots} 
 \end{figure}
 
 \begin{figure}[t]
 	\centering
 	{\includegraphics[width = 3.75 in]{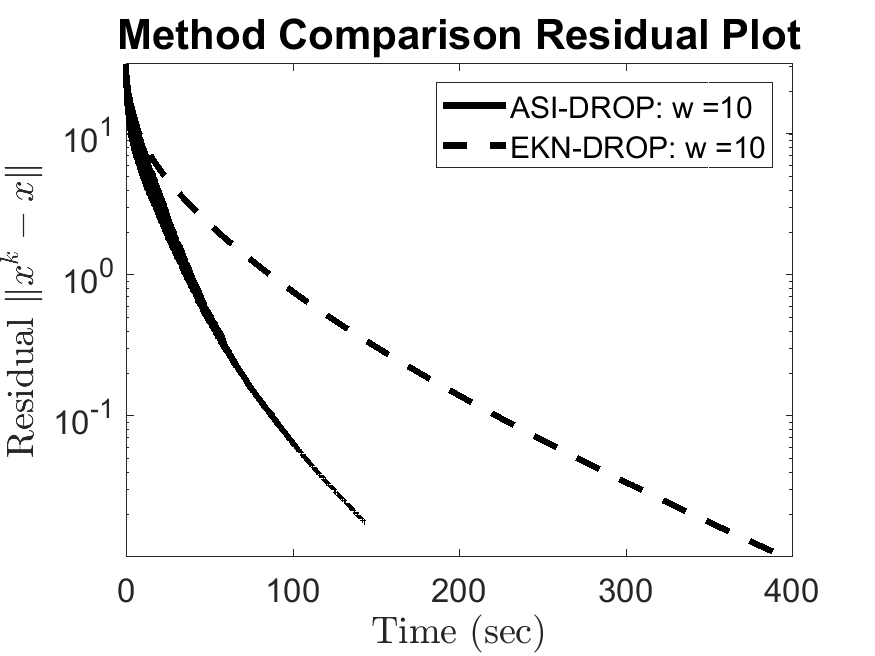}} 	 
 	\caption{Averages of residual plots over \numtrials trials for the ASI--DROP and EKN--DROP algorithms with $w=10$ slave nodes.}
 	\label{fig: residual-plots-div} 
 \end{figure}    	
 \begin{figure}[t]
 	\centering
 	\subfloat[ASI--DROP $(w=1)$]{\includegraphics[width = 1.5 in]{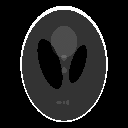}} 	
 	\hspace*{0.1 in} 
 	\subfloat[EKN--DROP $(w=1)$]{\includegraphics[width = 1.5 in]{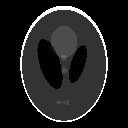}} 	 
 	\ \\	
 	\subfloat[ASI--DROP $(w=8)$]{\includegraphics[width = 1.5 in]{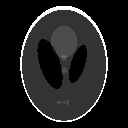}} 	
 	\hspace*{0.1 in} 
 	\subfloat[EKN--DROP $(w=8)$]{\includegraphics[width = 1.5 in]{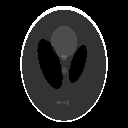}} 	 
 	\caption{Sample reconstructions.}
 	\label{fig: reconstructions} 
 \end{figure}

\bibliographystyle{spmpsci}       
\bibliography{bib-async-its} 

 \end{document}